\documentclass{article}
\usepackage[utf8]{inputenc}

\usepackage{amsmath, amsfonts, amssymb, amsthm, enumerate, geometry, tikz, hyperref, mathtools}
\usepackage{thmtools, thm-restate}
\usepackage{graphicx}
\usepackage{multirow}
\usepackage{import}
\usepackage[numbers]{natbib}
\usepackage{multicol}
\usepackage[nice]{units}
\usetikzlibrary{cd}
\tikzstyle{hackennode}=[draw,circle,inner sep=0,minimum size=4pt]
\tikzstyle{hackenline}=[line width=2pt]
\usetikzlibrary{arrows}
\usepackage{mathdots}
\usepackage{xcolor,colortbl}

\newtheorem{theorem}{Theorem}
\newtheorem{lemma}[theorem]{Lemma}
\newtheorem{prop}[theorem]{Proposition}
\newtheorem{cor}[theorem]{Corollary}

\newtheorem{remark}[theorem]{Remark}
\theoremstyle{definition}
\newtheorem{definition}[theorem]{Definition}
\newtheorem{example}[theorem]{Example}

\newcommand{\Z}{\mathbb{Z}}
\newcommand{\X}{\mathbb{X}}
\newcommand{\F}{\mathbb{F}}
\newcommand{\N}{\mathbb{N}}

\newcommand{\Hb}{\mathbb{H^+}}
\newcommand{\Hp}{\mathbb{H}_{\text{red}}(\Gamma)}

\title{Surgery Obstructions to Seifert Fibered Homology Spheres}
\author{Claire Zajaczkowski}

\begin{document}
\maketitle

\begin{abstract}
We examine surgeries on knots in $S^3$ to find surgery obstructions to Seifert fibered integral homology spheres. We find such obstructions by using Heegaard Floer and Knot Floer homology and the mapping cone formula for computing Heegaard Floer homology of surgery on a knot. Here however, we take a novel approach and use the toroidal structure, or the number of singular fibers of a homology sphere, to find obstructions. This approach allows us to show that genus one knots cannot yield Seifert fibered integral homology spheres with five or more singular fibers. Some other obstructions are also presented for higher genus knots. 
\end{abstract}

\section{Introduction} \label{introduction}
\subsection{Background}
Consider a knot $k$ in a closed oriented 3-manifold $M$. We can perform Dehn surgery on $k$ in $M$ by cutting $M$ open along the neighborhood of our knot, which is isomorphic to a solid torus, and gluing back in $D^2 \times S^1$ \cite{SavYellow}. This gluing is entirely determined by sending the meridian of $D^2 \times S^1$ to some simple closed curve, $pm + ql$, given by $p$ times the meridian and $q$ times the longitude of the knot exterior. This is called $\nicefrac{p}{q}$ surgery on the knot $k$. The space resulting from this construction is a closed oriented manifold, generally different from $M$. For this paper we will only be considering surgery on knots in $S^3$. Lickorish and Wallace showed that every closed orientable 3-manifold can be obtained by integral surgery on a link in $S^3$ \cite{Lickorish,Wallace}. Due to this result, Dehn surgery has become a fundamental method of representing 3-manifolds \cite{BoyerDS}. It is then natural to ask which manifolds can be represented by surgery on a knot in $S^3$, as opposed to a link.\par

Knots can be partitioned into three categories: torus, satellite and hyperbolic. Many surgery problems are understood for torus and satellite knots. For example, Moser completely classified surgery on torus knots in $S^3$ in \cite{MoserTK}, and Gabai was able to first prove the Property $P$ conjecture for satellite knots in \cite{Gabai}. (The Property $P$ conjectures has since been completely proved by Kronheimer and Mrowka \cite{PropP}.)  Since surgeries on torus and satellite knots are well understood, the most interesting surgeries to consider are those on hyperbolic knots. Surgeries on hyperbolic knots that do not yield hyperbolic manifolds are called exceptional. Exceptional surgeries are reducible, toroidal, or Seifert fibered. Thurston's hyperbolic Dehn surgery theorem says that there are only finitely many exceptional surgery slopes on a hyperbolic knot \cite{Thur}. We call Dehn surgery Seifert fibered, toroidal or reducible if it yields a Seifert fibered, toroidal or reducible manifold respectively \cite{TSFmont}. \par 

Since exceptional surgeries have remained the most elusive, there has been much work put into their study. Dean introduced a condition on knots in $S^3$ that guarantees a hyperbolic surgery \cite{DeansmallSF}, while Eudave-Mu\~noz extended this to include surgeries producing Seifert fibered manifolds with a projective plane orbit surface and two exceptional fibers \cite{EudMun}. In this same paper Eudave-Mu\~noz finds a collection of hyperbolic knots that yield toroidal Seifert fibered manifolds \cite[Proposition 4.5]{EudMun}. It can be checked that these knots will never yield Seifert fibered integral homology 3 spheres. Teragaito showed that any positive integers can arise as the toroidal surgery slope of a hyperbolic knot in \cite{TeraToroidal}, and Wu classified toroidal surgeries on length 3 Montesinos knots \cite{Wu1}.  For further general discussions of exceptional surgeries we suggest \cite{BritWu, IchiEx, Motegi}.\par

We are particularly interested in toroidal Seifert fibered integral homology 3-sphere surgeries, so let us explore these specifics further. Gordon and Luecke showed that the denominator of toroidal surgery slope is at most 2 for hyperbolic knots in \cite{GordonLueckeEssTori}, and Miyazaki and Motegi built on this result to show that if $K$ is a hyperbolic, periodic knot with period 2, only an integer coefficient can yield a toroidal surgery\cite{MiyMot}. Boyer and Zhang proved that toroidal Seifert fibered spaces cannot arise by non-integer surgery on  a hyperbolic knots in $S^3$, a result also implied by the work of Gordon and Luecke \cite{BZ, GL}.  Ichihara and Jong showed there is no toroidal Seifert fibered surgery on pretzel knots except for the trefoil in \cite{TSFmont}. Ozsv\'{a}th and Szab\'o showed that the family of Kinoshita-Terasaka knots $KR_{r,n}$ with $| r | \geq 2$ and $n \neq 0$ cannot yield integral Seifert fibered homology 3-spheres \cite{OSSurg}. Wu found the only three large arborescent knots that yield exceptional toroidal surgeries in \cite{WuExcep}. Wu also found the Montesinos knots that yield toroidal Seifert fibered surgeries in \cite{WuMont}. 

In this paper our goal is to find obstructions to which knots in $S^3$ yield integral Seifert fibered integral homology spheres. Seifert fibered, integral homology spheres are a result of exceptional surgery \cite{SavYellow}. One of the main tools used is Heegaard Floer homology, and the knot invariant: knot Floer homology. Heegaard Floer Homology is an invariant of closed 3 manifolds, introduced by Ozsv\'{a}th and Szab\'{o}. This invariant is isomorphic to Seiberg-Witten Floer Homology \cite{OSintro}. Ozsv\'{a}th and Szab\'{o} and many others have used Heegaard Floer homology to find such obstructions \cite{OSSurg, WuMont, WuExcep}. Here, however we use the number of singular fibers of a Seifert fibered integral homology sphere to find obstructions. This is using the toroidal structure of the manifold, as Seifert fibered homology spheres are toroidal if and only if they have 4 or more singular fibers. \par 

The other tools used in this paper are N\'{e}methi's graded root \cite{NemGR}, and Ozsv\'{a}th and Szab\'{o}'s mapping cone formula \cite{OSMC}. The graded root is a combinatorial object that allows us to compute Heegaard Floer homology more easily. This always exists when looking at Heegaard floer homology of Seifert fibered homology spheres, making it a very useful tool for our purposes. The mapping cone formula is used to make the Heegaard Floer homology of a manifold resulting from surgery more computible \cite{GainMC}. \par

Our general strategy is to examine the reduced Heegaard Floer homologies of both integral homology spheres and surgeries on knots. Specifically we observe how each of these behave under the U-action. We use Nem\'{e}thi's graded root to analyze this for Seifert fibered integral homology spheres and the mapping cone formula to analyze this for surgeries on knots. Then we compare these results in order to establish obstructions to surgeries on knots yielding Seifert Fibered integral homology spheres. \par

\subsection{Main Results}

Before we give our obstructions, we must introduce some notation. $HF^+(Y)$ is the Heegaard Floer homology of a 3-manifold $Y$; $HF_{red}(Y)$ is the reduced Heegaard Floer homology of $Y$. Also for the remainder of this paper we will only consider manifolds $Y$ that have negative definite plumbing. Our most general result is the following. 

\begin{restatable}[]{theorem}{HGres}
No surgery on a knot $K$ in $S^3$ of genus $g$ can yield a Seifert fibered integral  homology sphere with $l$ or more singular fibers if $g \leq \frac{24l - 103}{90}.$ 
\label{HGres}
\end{restatable}

The proof of this is easy to follow given the following general result, which is proved in Section 2. 

\begin{restatable}[]{lemma}{kcond}
Let $Y = \Sigma(p_1, p_2, \dots, p_l).$ Then if $$0 \leq k < \dfrac{1}{2}\left( l - 2 - \displaystyle \sum_{i=1}^l\dfrac{1}{p_i}\right),$$
$$U^k \cdot HF_{red}(Y) \neq 0.$$
\label{kcond}
\end{restatable}

\begin{proof}[Proof of Theorem \ref{HGres}] 
Gainullin's Theorem 3 in \cite{GainMC} tells us that $U^{g(K) +\left \lceil \frac{g_4(K)}{2}\right\rceil} \cdot HF_{red}\left(S^3_{p/q}(K)\right) = 0$ where $g(K)$ is the genus of the knot $K$, and $g_4(K)$ is the four ball genus of $K$. Using that $g_4(K) \leq g(K)$ we have the following:
\begin{align*}
g(K) + \left\lceil \frac{g_4(K)}{2} \right\rceil &\leq g(K) + \frac{g_4(K)+1}{2} \\
&\leq g(K) + \frac{g(K)+1}{2} \\
&= \frac{3g(K)+1}{2}.
\end{align*}
By Lemma \ref{kcond} we have that $U^k \cdot HF_{red}(Y) \neq 0$ for $$0 \leq k < \dfrac{1}{2}\left( l - 2 - \displaystyle \sum_{i=1}^l\dfrac{1}{p_i}\right),$$
thus if $$0 \leq \frac{3g(K)+1}{2} < \dfrac{1}{2}\left( l - 2 - \displaystyle \sum_{i=1}^l\dfrac{1}{p_i}\right),$$
we will have a contradiction, and $K$ will not yield a Seifert fibered integral  homology sphere with $l$ or more singular fibers in this case. Now we use that $$\displaystyle \sum_{i = 1}^l \frac{1}{p_i} < \displaystyle \sum_{i=1}^l\frac{1}{q_i} < \frac{1}{2} + \frac{1}{3} + \frac{l-2}{5},$$ where $q_i$ is the $i^{th}$ prime to solve the following inequality. 
\begin{align*}
\frac{3g(K)+1}{2} &< \dfrac{1}{2}\left( l - 2 - \displaystyle \sum_{i=1}^l\dfrac{1}{p_i}\right) \\
3g(K) + 1 &< l - 2 - \displaystyle \sum_{i = 1}^l\frac{1}{p_i}\\
\displaystyle \sum_{i = 1}^l \frac{1}{p_i} +3g(K) + 1 &< l - 2 \\
\displaystyle \sum_{i = 1}^l \frac{1}{p_i} &< l - 2 - 3g - 1 \\
\frac{1}{2} + \frac{1}{3} + \frac{l-2}{5} &< l - 3 - 3g(K) \\
g &< \frac{24l - 103}{90}.
\end{align*}
Thus our proof is complete. 
\end{proof}

We also have the following result for genus 1 knots. 

\begin{theorem}
No surgery on a genus 1 knot, $K$, in $S^3$ can yield a Seifert fibered integral homology sphere with 5 or more singular fibers. 
\label{gen1thm}
\end{theorem}

To prove this we need the following results: the first is proved in Section \ref{FiveSingularFibers}, the second is proved in Section \ref{MappingCone}. 

\begin{restatable}[]{theorem}{singfibersthm}
Let $Y = \Sigma(p_1,p_2,p_3,p_4,p_5).$ Then $$U\cdot HF_{red}(Y) \neq 0.$$
\label{5singfibersthm}
\end{restatable}

\begin{theorem}
For a genus 1 knot $K$ and $n \in \Z$ we have that $U \cdot HF_{red}(S^3_{1/n}(K)) = 0$.
\label{MCG1}
\end{theorem}

\begin{proof}[Proof of Theorem \ref{gen1thm}] By Lemma \ref{kcond} we have that $U\cdot HF_{red}(Y) \neq 0$ when $$1 < \frac{1}{2}\left(l - 2 - \displaystyle \sum_{i=1}^l\frac{1}{p_i}\right),$$ and $Y$ is a Seifert fibered integral homology sphere with $l$ of more singular fibers. We now show this condition is satisfied for $l \geq 6$.
\begin{align*}
\displaystyle \sum_{i=1}^6 \frac{1}{p_i}&< \frac{1}{2} + \frac{1}{3} + \frac{1}{5} + \frac{1}{7} + \frac{1}{11} + \frac{1}{13} \\
\displaystyle \sum_{i=1}^6 \frac{1}{p_i}&< 2\\
0 &< 2 - \displaystyle \sum_{i=1}^6 \frac{1}{p_i} \\
2 &< 4 - \displaystyle \sum_{i=1}^6 \frac{1}{p_i} \\
1 &< \frac{1}{2}\left( 6 - 2 - \displaystyle \sum_{i=1}^6 \frac{1}{p_i}\right)
\end{align*} 
Now by Theorem \ref{MCG1} we have that $U \cdot HF_{red}(S^3_{1/n}(K)) = 0$ for a genus 1 knot $K$. It follows that surgery on a genus 1 knot cannot yield a Seifert fibered integral homology sphere with 6 or more singular fibers. Then by applying Theorem \ref{5singfibersthm} our proof is complete. 

\end{proof}

\paragraph{Organization} In Section \ref{GradedRoots} we introduce N\'{e}methi's graded root and use this to prove Lemma \ref{kcond}. We also discuss some conjectures and the improved bound mentioned above. In Section \ref{FiveSingularFibers} we address the case of five singular fibers and prove Theorem \ref{5singfibersthm}. Finally, in Section \ref{MappingCone} we use the mapping cone formula for knot Floer homology to prove Theorem \ref{MCG1}. 

\paragraph{Acknowledgments} I would like to thank Tye Lidman for his guidance in research and Jen Hom and Duncan McCoy for their help with writing this paper. In particular, Duncan McCoy suggested a more streamlined way to prove the graded root results presented in Section 2. I would also like to thank Samuel Flynn for his help with the necessary coding involved in this work. I was partially supported by DMS-1709702 while working on this paper.

\section{Graded Roots}\label{GradedRoots}
In this section we introduce N\'{e}methi's graded root \cite{NemGR} and setup the preliminary results necessary for our obstructions. Originally this object was constructed from the plumbing graph for a plumbed 3-manifold. We will predominantly be following Can and Karakurt's specific construction for Seifert fibered homology spheres in \cite{CanKar}.

\subsection{Construction}
First, a graded root is defined as follows:
\begin{definition}(N\'emethi, \cite[Definition 3.1.2]{NemGR}) Let $R$ be an infinite tree with vertices $V$ and edges $E$. We denote by $[u, v]$ the edge with end-points $u$ and $v$. We say that $R$ is a \textbf{graded root} with grading $\chi \colon V \to \Z$ if
\begin{itemize}
\item[(a)] $\chi(u)- \chi(v) = \pm 1$ for any $[u,v] \in E$,
\item[(b)] $\chi(u) > \min\{\chi(v), \chi(w)\}$ for any $[u,v], [u,w] \in E$, and $v \neq w$,
\item[(c)] $\chi$ is bounded below, and $\chi^{-1}(k)$ is finite for any $k \in \Z$ and $| \chi^{-1}(k)| = 1$ for $k$ sufficiently large. 
\end{itemize}
\end{definition}

Now we can give the basics for the construction of a graded root for a given Seifert fibered homology sphere, $Y = \Sigma(p_1,\dots, p_l)$, where $p_1<p_2<\cdots<p_l$. We define a function $\Delta \colon \N \to \Z$ by 
\begin{align}
\Delta(n) = 1+|e_0|n - \displaystyle \sum_{i=1}^l \left\lceil \frac{np_i'}{p_i}\right\rceil,
   \label{Delta}
 \end{align}
  where $(e_0, p_1',\dots,p_l')$ is the unique solution to 
  $$e_0p_1\cdots p_l + p_1'p_2\cdots p_l + p_1p_2'\cdots p_l +\cdots + p_1p_2 \cdots p_l' = -1,$$ and $$0 < p_i' \leq p_i - 1 \text{ for } i =1, 2, \dots, l.$$ By \cite[Theorem 4.1, (2)]{CanKar} we have that $\Delta(n)$ is always positive for $n \geq N_0$ and 
  $$N_0 = p_1p_2 \cdots p_l \left( (l-2) - \displaystyle \sum_{i=1}^l \frac{1}{p_i}\right).$$

The general definition follows. 
  \begin{definition}(Karakurt-Lidman, \cite[Definition 3.1]{KarLid})
  A \textbf{$\Delta$-sequence} is a pair $(X, \delta)$ where $X$ is a well ordered finite set, and $\delta \colon X \to \Z \setminus \{0\}$ with $\delta(x_0) > 0$ where $x_0$ is the minimum of $X$. 
  \end{definition}
  
For our purposes, we will have the $\Delta$-sequence $(X, \Delta)$, where $\Delta(n)$ is given by \eqref{Delta} and $X = \{x \in \N \mid \Delta(x) \neq 0 \text{ and } x \leq N_0\}$. We also have that $\Delta$ will be symmetric, as $\Delta(n) = -\Delta(N_0-n)$ \cite[Theorem 4.1, (2)]{CanKar}. For the remainder of this paper when we refer to a $\Delta$-sequence we will be referring to the set $\{\Delta(x_0), \Delta(x_1), \dots \Delta(x_k)\}$, where $x_k$ is the last integer that satisfies $\Delta(x_k) <0$, as we will only be using the function defined in \eqref{Delta}.\par
  
Once we have our $\Delta$-sequence, we can define a $\tau$ function $\tau_{\Delta} \colon \{0,1,\dots,k\} \to \Z$ using the recurrence relation $\tau_\Delta(n+1) - \tau_\Delta(n) = \Delta(x_n)$, with the initial condition $\tau_\Delta(0) = 0$. Let $\{\tau_\Delta(0), \dots, \tau_\Delta(k)\}$, be called the $\tau$ sequence. \par

Now that we know how to get a $\tau$ sequence for a Seifert fibered homology sphere, we can use this to construct our graded root, following \cite[Example 3.1.3]{NemGR}. For every $n \in \Z$, let $R_n$ be the infinite graph with vertex set $\Z \cap [\tau(n), \infty)$ and edge set $\{[k,k+1] \colon k \in \Z \cap [\tau(n),\infty)\}$. We then identify all common vertices and edges of $R_n$ and $R_{n+1}$. This gets us an infinite tree $\Gamma_\tau$, and we assign a function $\chi(v)$ that gives the unique integer corresponding to the grading of the vertex $v$. 

\begin{example}
The construction of $\Gamma_\tau$ for the $\tau$ sequence $\{0,-1,0,-1,-2,-3,-4,-3\}$, is seen in Figure \ref{taufig}.
 \begin{figure}[h]
 \centering
 \label{taufig}
\begin{tikzpicture}[thick,scale=0.6, every node/.style={scale=0.6}]
\node[hackennode] (w0) at (-4,0) {};
\node[left] at (-4,.25) {};
\node[hackennode] (w1) at (-4,1) {};
\node[left] at (-4,1.25) {};
\node[hackennode] (w2) at (-4,2) {};
\node[left] at (-4,2.25) {};
\node[hackennode] (w3) at (-4,3) {};
\node[hackennode] (w3.0) at (-4,3.25) {};
\node[hackennode] (w3.1) at (-4,3.5) {};
\node[below] at (-4,4.25) {$R_0$};

\draw[thick, blue] (w0) -- (w1) -- (w2) -- (w3);

\node[hackennode] (v-1) at (-3,-1) {};
\node[left] at (-3,-.75) {};
\node[hackennode] (v0) at (-3,0) {};
\node[hackennode] (v1) at (-3,1) {};
\node[hackennode] (v2) at (-3,2) {};
\node[hackennode] (v3) at (-3,3) {};
\node[hackennode] (v3.0) at (-3,3.25) {};
\node[hackennode] (v3.1) at (-3,3.5) {};
\node[below] at (-3,4.25) {$R_1$};

\draw[thick, blue] (v-1) -- (v0) -- (v1) -- (v2) -- (v3);

\node[hackennode] (x0) at (-2,0) {};
\node[hackennode] (x1) at (-2,1) {};
\node[hackennode] (x2) at (-2,2) {};
\node[hackennode] (x3) at (-2,3) {};
\node[hackennode] (x3.0) at (-2,3.25) {};
\node[hackennode] (x3.1) at (-2,3.5) {};
\node[below] at (-2,4.25) {$R_2$};

\draw[thick, blue] (x0) -- (x1) -- (x2) -- (x3);

\node[hackennode] (y-1) at (-1,-1) {};
\node[left] at (-1,-.75) {};
\node[hackennode] (y0) at (-1,0) {};
\node[hackennode] (y1) at (-1,1) {};
\node[hackennode] (y2) at (-1,2) {};
\node[hackennode] (y3) at (-1,3) {};
\node[hackennode] (y3.0) at (-1,3.25) {};
\node[hackennode] (y3.1) at (-1,3.5) {};
\node[below] at (-1,4.25) {$R_3$};

\draw[thick, blue] (y-1) -- (y0) -- (y1) -- (y2) -- (y3);

\node[hackennode] (z-2) at (0,-2) {};
\node[left] at (0,-1.75) {};
\node[hackennode] (z-1) at (0,-1) {};
\node[hackennode] (z0) at (0,0) {};
\node[hackennode] (z1) at (0,1) {};
\node[hackennode] (z2) at (0,2) {};
\node[hackennode] (z3) at (0,3) {};
\node[hackennode] (z3.0) at (0,3.25) {};
\node[hackennode] (z3.1) at (0,3.5) {};
\node[below] at (-0,4.25) {$R_4$};

\draw[thick, blue] (z-2) -- (z-1) -- (z0) -- (z1) -- (z2) -- (z3);

\node[hackennode] (a-3) at (1,-3) {};
\node[left] at (1,-2.75) {};
\node[hackennode] (a-2) at (1,-2) {};
\node[hackennode] (a-1) at (1,-1) {};
\node[hackennode] (a0) at (1,0) {};
\node[hackennode] (a1) at (1,1) {};
\node[hackennode] (a2) at (1,2) {};
\node[hackennode] (a3) at (1,3) {};
\node[hackennode] (a3.0) at (1,3.25) {};
\node[hackennode] (a3.1) at (1,3.5) {};
\node[below] at (1,4.25) {$R_5$};

\draw[thick, blue] (a-3) -- (a-2) -- (a-1) -- (a0) -- (a1) -- (a2) -- (a3);

\node[hackennode] (b-4) at (2,-4) {};
\node[left] at (2,-3.75) {};
\node[hackennode] (b-3) at (2,-3) {};
\node[hackennode] (b-2) at (2,-2) {};
\node[hackennode] (b-1) at (2,-1) {};
\node[hackennode] (b0) at (2,0) {};
\node[hackennode] (b1) at (2,1) {};
\node[hackennode] (b2) at (2,2) {};
\node[hackennode] (b3) at (2,3) {};
\node[hackennode] (b3.0) at (2,3.25) {};
\node[hackennode] (b3.1) at (2,3.5) {};
\node[below] at (2,4.25) {$R_6$};

\draw[thick, blue] (b-4) -- (b-3) -- (b-2) -- (b-1) -- (b0) -- (b1) -- (b2) -- (b3);

\node[hackennode] (c-3) at (3,-3) {};
\node[hackennode] (c-2) at (3,-2) {};
\node[hackennode] (c-1) at (3,-1) {};
\node[hackennode] (c0) at (3,0) {};
\node[hackennode] (c1) at (3,1) {};
\node[hackennode] (c2) at (3,2) {};
\node[hackennode] (c3) at (3,3) {};
\node[hackennode] (c3.0) at (3,3.25) {};
\node[hackennode] (c3.1) at (3,3.5) {};
\node[below] at (3,4.25) {$R_7$};

\draw[thick, blue] (c-3) -- (c-2) -- (c-1) -- (c0) -- (c1) -- (c2) -- (c3);

\node[left] at (4,-.5) {$\rightarrow$};

\node[hackennode] (w0') at (5,0) {};
\node[left] at (5,.25) {};
\node[hackennode] (w1') at (5,1) {};
\node[left] at (5,1.25) {};
\node[hackennode] (w2') at (5,2) {};
\node[left] at (5,2.25) {};
\node[hackennode] (w3') at (5,3) {};
\node[hackennode] (w3.0') at (5,3.25) {};
\node[hackennode] (w3.1') at (5,3.5) {};
\node[below] at (5,4.25) {$\Gamma_\tau$};

\node[hackennode] (v1') at (5,-1) {};
\node[left] at (5,-.75) {};
\node[hackennode] (v3') at (6,-1) {};
\node[right] at (6,-.75) {};
\node[hackennode] (v4') at (6,-2) {};
\node[right] at (6,-1.75) {};
\node[hackennode] (v5') at (6,-3) {};
\node[right] at (6,-2.75) {};
\node[hackennode] (v6') at (6,-4) {};
\node[right] at (6,-3.75) {};

\draw[thick, blue] (v1') -- (w0') -- (w1') -- (w2') -- (w3');
\draw[thick, blue] (w0') -- (v3') -- (v4') -- (v5') -- (v6');

\node[left] at (-5,2) {2};
\node[left] at (-5,1) {1};
\node[left] at (-5,0) {0};
\node[left] at (-5,-1) {-1};
\node[left] at (-5,-2) {-2};
\node[left] at (-5,-3) {-3};
\node[left] at (-5,-4) {-4};

\draw[gray, dashed] (-5,2)--(8,2);
\draw[gray, dashed] (-5,1)--(8,1);
\draw[gray, dashed] (-5,0)--(8,0);
\draw[gray, dashed] (-5,-1)--(8,-1);
\draw[gray, dashed] (-5,-2)--(8,-2);
\draw[gray, dashed] (-5,-3)--(8,-3);
\draw[gray, dashed] (-5,-4)--(8,-4);

\end{tikzpicture}

 \caption{Construction of $\Gamma_\tau$}
\end{figure}
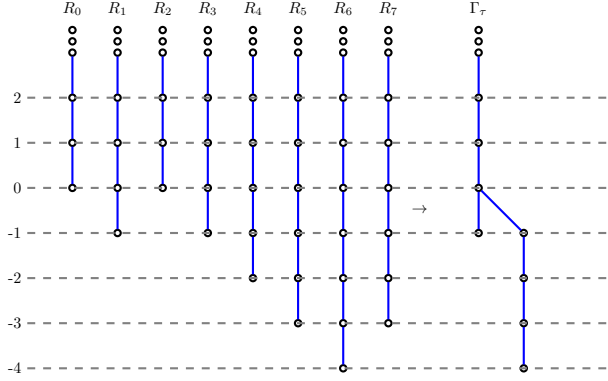
\end{example}

We now define some operations on $\Delta$-sequences which will be used to prove our results. 

\begin{definition}(Karakurt-Lidman, \cite[Section 3]{KarLid})
Let $(X, \Delta)$ be a $\Delta$-sequence. Let $t$ be a positive integer and $x \in X$ with $| \Delta(x)| \geq t$. From this we construct a new $\Delta$-sequence $(X', \Delta')$ as follows. The set $X'$ is obtained by removing $x$ from $X$ and putting $t$ consecutive elements $x_1,\dots,x_t$ in its place. Now, choose nonzero integers $n_1, \dots , n_t$ each with the same sign as $\Delta(x)$ such that $n_1 + \cdots + n_t = \Delta(x)$. The new $\Delta'$ agrees with $\Delta$ on $X \setminus \{x\}$ and satisfies $\Delta'(x_i) = n_i$ for $i = 1,\dots,t$. $(X',\Delta')$ is called a \textbf{refinement} of $(X,\Delta)$ at $x$. Conversely $(X, \Delta)$ is called a \textbf{merge} of $(X', \Delta')$
\end{definition}

\subsection{Properties}

Now that we can construct our graded root, we will introduce some properties and terminology. To any graded root $\Gamma$, we associate a $\Z$-graded, $\Z[U]$-module. We let $\Hb(\Gamma)$ be the free abelian group generated by the vertex set of $\Gamma$, and the degree of the generator corresponding to the vertex, $v$ has degree $2 \chi(v)$. The $U$-action on $\Hb ( \Gamma)$ is a degree $-2$ endomorphism that sends each vertex $v$ to the sum of all vertices $w$ connected to $v$ by an edge where $\chi(w) < \chi(x)$. If no vertices $w$ satisfy this, then $U$ sends $v$ to zero. We can now define two finitely generated groups:
$$\Hp = \operatorname{Coker}(U^n), \text{ for large } n,$$
$$\widehat{\mathbb{H}}(\Gamma) = \operatorname{Ker}(U) \oplus \operatorname{Coker}(U)[-1].$$
We will also refer to the image of $U^n$ for large $n$ as the tower, denoted $\mathcal{T}^+$, which will eventually stabilize and become constant. We can then view $\Hp$, which is the primary focus of this paper, as $\nicefrac{\Hb(\Gamma)}{\mathcal{T}^+}$. \par
These groups encode Heegaard Floer homology as follows:
\begin{theorem}(N\'emethi, \cite[Section 11.3]{NemOS})
Let $Y$ be a positively oriented integral homology sphere, then we have the following isomorphisms of $\Z[U]$-modules up to an overall degree shift:
\begin{itemize}
\item[(1)] $HF^+(-Y) \cong \Hb(\Gamma)$,
\item[(2)] $HF_{red}(-Y) \cong \Hp$,
\item[(3)] $\widehat{HF}(-Y) \cong \widehat{\mathbb{H}}(\Gamma)$.
\end{itemize}
\end{theorem}

\begin{prop}(Karakurt-Lidman, \cite[Proposition 3.6]{KarLid})
Refinements and merges do not change $\Hb(\Gamma)$ or $\Hp$. 
\end{prop}

\subsection{Results}

We are now ready to prove Lemma \ref{DeltaCond}

\begin{lemma}
If a $\Delta$-sequence for a graded root $\Gamma$ contains $x$ such that $\Delta(x) = k+1$ and $x'>x$ such that $\Delta(x') = -(k+1)$
then $$U^k \cdot \mathbb{H}_{red}(\Gamma) \neq 0.$$
\label{DeltaCond}
\end{lemma}

\begin{proof}
Since $\Delta(x) = k+1$ we have $$\tau_{\Delta}(i+1) - \tau_{\Delta}(i) = k+1$$ where $x$ is in the $i^{th}$ position in our $\Delta$-sequence. Similarly since $\Delta(x') = -(k+1)$ we have $$\tau_{\Delta}(j) - \tau_{\Delta}(j+1) = k+1$$ where $x'$ is in the $j^{th}$ position in our $\Delta$-sequence. Since we assumed $x' > x$, we know that $j>i$. It is also clear that $$\tau_{\Delta}(i+1) > \tau_{\Delta}(i) \quad \text{ and } \quad \tau_{\Delta}(j) > \tau_{\Delta}(j+1).$$ We now have two cases. First let $\tau_{\Delta}(i) > \tau_{\Delta}(j+1)$. This is shown in Figure \ref{Case1}. We have that $$U^k \cdot a = b$$ in this case, and since $a$ and $c$ are both at grading $n+k$, $a$ is not in the image of $U^z$ for large $z$. Thus $a \in \mathbb{H}_{red}(\Gamma)$ and $U^k \cdot a \neq 0$ so $$U^k \cdot \mathbb{H}_{red}(\Gamma) \neq 0.$$

Now consider when $\tau_\Delta(i) > \tau_\Delta(j+1)$, which is shown in Figure \ref{Case2}. We can proceed exactly as above to find that  $$U^k \cdot \mathbb{H}_{red}(\Gamma) \neq 0.$$ Thus our proof is complete. 

\begin{figure}
  \begin{minipage}{.5\textwidth}
  \centering
\begin{tikzpicture}[thick,scale=0.5, every node/.style={scale=0.5}]
%Grading
    %Top Dots
    \node[hackennode] (dot1) at (-6,4.25) {};
    \node[hackennode] (dot1) at (-6,4.5) {};
    \node[hackennode] (dot1) at (-6,4.75) {};
    %n+k+3
    \node at (-6,3) {$n+k+3$};
    \draw[gray, dashed] (-5,3)--(5,3);
    %n+k+2
    \node at (-6,2) {$n+k+2$};
    \draw[gray, dashed] (-5,2)--(5,2);
    %n+k+1
    \node at (-6,1) {$n+k+1$};
    \draw[gray, dashed] (-5,1)--(5,1);
    %n+k
    \node at (-6,0) {$n+k$};
    \draw[gray, dashed] (-5,0)--(5,0);
    %Upper Center Dots
    \node[hackennode] (dot1) at (-6,-1.25) {};
    \node[hackennode] (dot1) at (-6,-1.5) {};
    \node[hackennode] (dot1) at (-6,-1.75) {};
    %n+1
    \node at (-6,-3) {$n+1$};
    \draw[gray, dashed] (-5,-3)--(5,-3);
    %n
    \node at (-6,-4) {$n$};
    \draw[gray, dashed] (-5,-4)--(5,-4);
    %Lower Center Dots
    \node[hackennode] (dot1) at (-6,-5.75) {};
    \node[hackennode] (dot1) at (-6,-5.5) {};
    \node[hackennode] (dot1) at (-6,-5.25) {};
    %m+1
    \node at (-6,-7) {$m+1$};
    \draw[gray, dashed] (-5,-7)--(5,-7);
    %m
    \node at (-6,-8) {$m$};
    \draw[gray, dashed] (-5,-8)--(5,-8);
    %Bottom Dots
    \node[hackennode] (dot1) at (-6,-9.75) {};
    \node[hackennode] (dot1) at (-6,-9.5) {};
    \node[hackennode] (dot1) at (-6,-9.25) {};

%Left Dots
\node[hackennode] (Ldot1) at (-4.25,-1.5) {};
\node[hackennode] (Ldot2) at (-4,-1.5) {};
\node[hackennode] (Ldot3) at (-4.5,-1.5) {};

%Middle Dots
\node[hackennode] (Mdot1) at (-.25,-1.5) {};
\node[hackennode] (Mdot2) at (0,-1.5) {};
\node[hackennode] (Mdot3) at (.25,-1.5) {};

%Right Dots
\node[hackennode] (Rdot1) at (3.75,-1.5) {};
\node[hackennode] (Rdot2) at (4,-1.5) {};
\node[hackennode] (Rdot3) at (4.25,-1.5) {};

%Tau i - Called A
    %Nodes
    \node[hackennode] (An) at (-3,-4) {};
    \node[hackennode] (An+1) at (-3,-3) {};
    %Lower Dots
    \node[hackennode] (Da) at (-3,-2.5) {};
    \node[hackennode] (Db) at (-3,-1.25) {};
    \node[hackennode] (Dc) at (-3,-1.5) {};
    \node[hackennode] (Dd) at (-3,-1.75) {};
    \node[hackennode] (De) at (-3,-.5) {};
    %Nodes
    \node[hackennode] (An+k) at (-3,0) {};
    \node[hackennode] (An+k+1) at (-3,1) {};
    \node[hackennode] (An+k+2) at (-3,2) {};
    \node[hackennode] (An+k+3) at (-3,3) {};
    %Top Dots
    \node[hackennode] (Df) at (-3,3.5) {};
    \node[hackennode] (Dg) at (-3,4.75) {};
    \node[hackennode] (Dh) at (-3,4.5) {};
    \node[hackennode] (Di) at (-3,4.25) {};
    %Lines
    \draw[very thick, blue] (An) -- (An+1) -- (Da);
    \draw[very thick, blue] (De) -- (An+k) -- (An+k+1) -- (An+k+2) -- (An+k+3) -- (Df) ;
    %Labels
    \node[below] (Lj+1) at (-3,-4.25) {$\tau_{\Delta}(i)$};
    \node[left] (a) at (-3,.25) {\textbf{a}};
    \node[left] (a) at (-3,-3.75) {\textbf{b}};
    
%Tau i+1 - Called B
    %Nodes
    \node[hackennode] (Bn+k+1) at (-1.5,1) {};
    \node[hackennode] (Bn+k+2) at (-1.5,2) {};
    \node[hackennode] (Bn+k+3) at (-1.5,3) {};
    %Top Dots
    \node[hackennode] (Bf) at (-1.5,3.5) {};
    \node[hackennode] (Bg) at (-1.5,4.75) {};
    \node[hackennode] (Bh) at (-1.5,4.5) {};
    \node[hackennode] (Bi) at (-1.5,4.25) {};
    %Lines
    \draw[very thick, blue] (Bn+k+1) -- (Bn+k+2) -- (Bn+k+3) -- (Bf);
    %Label
    \node[below] (Lj+1) at (-1.5,.75) {$\tau_{\Delta}(i+1)$};
    
%Tau j+1 - Called C
    %Nodes
    \node[hackennode] (Cm) at (2,-8) {};
    \node[hackennode] (Cm+1) at (2,-7) {};
    %Lower Dots
    \node[hackennode] (Cv) at (2,-6.5) {};
    \node[hackennode] (Cw) at (2,-5.75) {};
    \node[hackennode] (Cx) at (2,-5.5) {};
    \node[hackennode] (Cy) at (2,-5.25) {};
    \node[hackennode] (Cz) at (2,-4.5) {};
    %Nodes
    \node[hackennode] (Cn) at (2,-4) {};
    \node[hackennode] (Cn+1) at (2,-3) {};
    %Middle Dots
    \node[hackennode] (Ca) at (2,-2.5) {};
    \node[hackennode] (Cb) at (2,-1.25) {};
    \node[hackennode] (Cc) at (2,-1.5) {};
    \node[hackennode] (Cd) at (2,-1.75) {};
    \node[hackennode] (Ce) at (2,-.5) {};
    %Nodes
    \node[hackennode] (Cn+k) at (2,0) {};
    \node[hackennode] (Cn+k+1) at (2,1) {};
    \node[hackennode] (Cn+k+2) at (2,2) {};
    \node[hackennode] (Cn+k+3) at (2,3) {};
    %Top Dots
    \node[hackennode] (Cf) at (2,3.5) {};
    \node[hackennode] (Cg) at (2,4.75) {};
    \node[hackennode] (Ch) at (2,4.5) {};
    \node[hackennode] (Ci) at (2,4.25) {};
    %Lines
    \draw[very thick, blue] (Cz) -- (Cn) -- (Cn+1) -- (Ca);
    \draw[very thick, blue] (Ce) -- (Cn+k) -- (Cn+k+1) -- (Cn+k+2) -- (Cn+k+3) -- (Cf) ;
    \draw[very thick, blue] (Cm) -- (Cm+1) -- (Cv);
    %Label
    \node[below] (Lj+1) at (2,-8.25) {$\tau_{\Delta}(j+1)$};
    \node[left] (c) at (2,.25) {\textbf{c}};
\end{tikzpicture}
\caption{$\tau_\Delta(i) > \tau_\Delta(j+1)$}
    \label{Case1}
\end{minipage}
\begin{minipage}{.5\textwidth}
\centering
\begin{tikzpicture}[thick,scale=0.5, every node/.style={scale=0.5}]
%Grading
    %Top Dots
    \node[hackennode] (dot1) at (-6,4.25) {};
    \node[hackennode] (dot1) at (-6,4.5) {};
    \node[hackennode] (dot1) at (-6,4.75) {};
    %n+k+3
    \node at (-6,3) {$n+k+3$};
    \draw[gray, dashed] (-5,3)--(5,3);
    %n+k+2
    \node at (-6,2) {$n+k+2$};
    \draw[gray, dashed] (-5,2)--(5,2);
    %n+k+1
    \node at (-6,1) {$n+k+1$};
    \draw[gray, dashed] (-5,1)--(5,1);
    %n+k
    \node at (-6,0) {$n+k$};
    \draw[gray, dashed] (-5,0)--(5,0);
    %Upper Center Dots
    \node[hackennode] (dot1) at (-6,-1.25) {};
    \node[hackennode] (dot1) at (-6,-1.5) {};
    \node[hackennode] (dot1) at (-6,-1.75) {};
    %n+1
    \node at (-6,-3) {$n+1$};
    \draw[gray, dashed] (-5,-3)--(5,-3);
    %n
    \node at (-6,-4) {$n$};
    \draw[gray, dashed] (-5,-4)--(5,-4);
    %Lower Center Dots
    \node[hackennode] (dot1) at (-6,-5.75) {};
    \node[hackennode] (dot1) at (-6,-5.5) {};
    \node[hackennode] (dot1) at (-6,-5.25) {};
    %m+1
    \node at (-6,-7) {$m+1$};
    \draw[gray, dashed] (-5,-7)--(5,-7);
    %m
    \node at (-6,-8) {$m$};
    \draw[gray, dashed] (-5,-8)--(5,-8);
    %Bottom Dots
    \node[hackennode] (dot1) at (-6,-9.75) {};
    \node[hackennode] (dot1) at (-6,-9.5) {};
    \node[hackennode] (dot1) at (-6,-9.25) {};

%Left Dots
\node[hackennode] (Ldot1) at (-4.25,-1.5) {};
\node[hackennode] (Ldot2) at (-4,-1.5) {};
\node[hackennode] (Ldot3) at (-4.5,-1.5) {};

%Middle Dots
\node[hackennode] (Mdot1) at (-.25,-1.5) {};
\node[hackennode] (Mdot2) at (0,-1.5) {};
\node[hackennode] (Mdot3) at (.25,-1.5) {};

%Right Dots
\node[hackennode] (Rdot1) at (4.75,-1.5) {};
\node[hackennode] (Rdot2) at (4.5,-1.5) {};
\node[hackennode] (Rdot3) at (4.25,-1.5) {};

%Tau j+1 - Called A
    %Nodes
    \node[hackennode] (An) at (3,-4) {};
    \node[hackennode] (An+1) at (3,-3) {};
    %Lower Dots
    \node[hackennode] (Da) at (3,-2.5) {};
    \node[hackennode] (Db) at (3,-1.25) {};
    \node[hackennode] (Dc) at (3,-1.5) {};
    \node[hackennode] (Dd) at (3,-1.75) {};
    \node[hackennode] (De) at (3,-.5) {};
    %Nodes
    \node[hackennode] (An+k) at (3,0) {};
    \node[hackennode] (An+k+1) at (3,1) {};
    \node[hackennode] (An+k+2) at (3,2) {};
    \node[hackennode] (An+k+3) at (3,3) {};
    %Top Dots
    \node[hackennode] (Df) at (3,3.5) {};
    \node[hackennode] (Dg) at (3,4.75) {};
    \node[hackennode] (Dh) at (3,4.5) {};
    \node[hackennode] (Di) at (3,4.25) {};
    %Lines
    \draw[very thick, blue] (An) -- (An+1) -- (Da);
    \draw[very thick, blue] (De) -- (An+k) -- (An+k+1) -- (An+k+2) -- (An+k+3) -- (Df) ;
    %Labels
    \node[below] (Lj+1) at (3,-4.25) {$\tau_{\Delta}(j+1)$};
    \node[left] (a) at (3,.25) {\textbf{a}};
    \node[left] (a) at (3,-3.75) {\textbf{b}};
    
%Tau j - Called B
    %Nodes
    \node[hackennode] (Bn+k+1) at (1.5,1) {};
    \node[hackennode] (Bn+k+2) at (1.5,2) {};
    \node[hackennode] (Bn+k+3) at (1.5,3) {};
    %Top Dots
    \node[hackennode] (Bf) at (1.5,3.5) {};
    \node[hackennode] (Bg) at (1.5,4.75) {};
    \node[hackennode] (Bh) at (1.5,4.5) {};
    \node[hackennode] (Bi) at (1.5,4.25) {};
    %Lines
    \draw[very thick, blue] (Bn+k+1) -- (Bn+k+2) -- (Bn+k+3) -- (Bf);
    %Label
    \node[below] (Lj+1) at (1.5,.75) {$\tau_{\Delta}(j)$};
    
%Tau i - Called C
    %Nodes
    \node[hackennode] (Cm) at (-2,-8) {};
    \node[hackennode] (Cm+1) at (-2,-7) {};
    %Lower Dots
    \node[hackennode] (Cv) at (-2,-6.5) {};
    \node[hackennode] (Cw) at (-2,-5.75) {};
    \node[hackennode] (Cx) at (-2,-5.5) {};
    \node[hackennode] (Cy) at (-2,-5.25) {};
    \node[hackennode] (Cz) at (-2,-4.5) {};
    %Nodes
    \node[hackennode] (Cn) at (-2,-4) {};
    \node[hackennode] (Cn+1) at (-2,-3) {};
    %Middle Dots
    \node[hackennode] (Ca) at (-2,-2.5) {};
    \node[hackennode] (Cb) at (-2,-1.25) {};
    \node[hackennode] (Cc) at (-2,-1.5) {};
    \node[hackennode] (Cd) at (-2,-1.75) {};
    \node[hackennode] (Ce) at (-2,-.5) {};
    %Nodes
    \node[hackennode] (Cn+k) at (-2,0) {};
    \node[hackennode] (Cn+k+1) at (-2,1) {};
    \node[hackennode] (Cn+k+2) at (-2,2) {};
    \node[hackennode] (Cn+k+3) at (-2,3) {};
    %Top Dots
    \node[hackennode] (Cf) at (-2,3.5) {};
    \node[hackennode] (Cg) at (-2,4.75) {};
    \node[hackennode] (Ch) at (-2,4.5) {};
    \node[hackennode] (Ci) at (-2,4.25) {};
    %Lines
    \draw[very thick, blue] (Cz) -- (Cn) -- (Cn+1) -- (Ca);
    \draw[very thick, blue] (Ce) -- (Cn+k) -- (Cn+k+1) -- (Cn+k+2) -- (Cn+k+3) -- (Cf) ;
    \draw[very thick, blue] (Cm) -- (Cm+1) -- (Cv);
    %Label
    \node[below] (Lj+1) at (-2,-8.25) {$\tau_{\Delta}(i)$};
    \node[left] (c) at (-2,.25) {\textbf{c}};
\end{tikzpicture}
\caption{$\tau_\Delta(i) < \tau_\Delta(j+1)$}
    \label{Case2}
\end{minipage}
\end{figure}
\end{proof}

Lemma \ref{DeltaCond} is equivalent to requiring a sequence of values $(x_1, \dots, x_j)$ and $(x'_1,\dots,x'_l)$ that after the appropriate refinements or merges we have $$\{\dots, k+1, \dots, -(k+1),\dots\}$$ in our $\Delta$-sequence.
There is no need to prove this separately as we know refinements and merges do not change $\mathbb{H}_{red}$. \par

The next Lemma follows easily. 

\kcond*

\begin{proof}
Consider $n = kp_1p_2\cdots p_l$. We have that 
\begin{align*}
\Delta(n) &= 1+|e_0|n- \displaystyle \sum_{i=1}^l\dfrac{np_i'}{p_i}\\
&= 1 + |e_0|n - k \cdot \left( p_1'p_2\cdots p_l + p_1p_2'\cdots p_l + \cdot + p_1p_2\cdots p_l'\right)\\
&= 1+ |e_0|n - k\left(-1 - e_0p_1\cdots p_l\right)\\
&= 1 + k +|e_0|kp_1\cdots p_l + ke_0p_1\cdots p_l\\
&= k+1,
\end{align*}
using substitution from the Diophantine equation.
By using that $\Delta(x) = -\Delta(N_0-x)$ we have that $$\Delta(N_0 - n) = -(k+1).$$ Now using our assumptions on $k$ we have the following,
\begin{align*}
    n &= kp_1\cdots p_l \\
    n&< \dfrac{1}{2}\left( - 2 - \displaystyle \sum_{i=1}^l\dfrac{1}{p_i}\right)p_1\cdots p_l \\
    n&<\dfrac{1}{2}N_0\\
    2n &< N_0\\
    n&< N_0 - n.
\end{align*}
Thus $\Delta(n) = k+1$, $\Delta(N_0 - n) = -(k+1)$ and $N_0 - n > n$ and so the conditions of Lemma \ref{DeltaCond} are satisfied and $U^k\cdot HF_{red}(Y) \neq 0$. 
\end{proof}

\section{Five Singular Fibers}\label{FiveSingularFibers}

The goal of this section is to prove Theorem \ref{5singfibersthm}. Our strategy is to find a sequence of values $(x, x+1, x+2, \dots, x+j)$, where $x+j < \nicefrac{N_0}{2}$ such that $$(\Delta(x), \Delta(x+1), \dots, \Delta(x+j-1), \Delta(x+j) = (1, 0, \dots, 0, 1).$$ Then by \eqref{DeltaCond} we have that 
$$(\Delta(N_0 - x), \Delta(N_0 - x+1), \dots, \Delta(N_0 - x+j-1), \Delta(N_0 - x+j) = (-1, 0, \dots, 0, -1).$$
By performing the appropriate merges we have a $\Delta$-sequence as follows
$$\{\dots, 2, \dots, -2\}.$$
Then the proof of Theorem \ref{5singfibersthm} follows from Lemma \ref{DeltaCond}.\par 

\vspace{5mm}

First consider the following corollary to Lemma \ref{kcond}.

\begin{cor}
Let $Y = \Sigma(p_1,p_2,p_3,p_4,p_5).$ Then $$U\cdot HF_{red}(Y) \neq 0,$$ when one of the following holds,
\begin{enumerate}
\item $p_1 \geq 4$,
\item $p_1 = 3$ and $p_5 \geq 17$,
\item $p_1 = 2$, $p_2 = 3$, and $p_3 \geq 17$,
\item $p_1 = 2$, $ p_2 = 3$, $p_3 = 7$ and $p_4 \geq 83$,
\item $p_1 = 2$, $p_2 = 3$, $p_3 = 7$, $p_4 = 43$, $p_5 \geq 1811$,
\item $p_1 = 2$, $p_2 = 3$, $p_3 = 11$, $p_4 \geq 15$ and $p_5 \geq 101$.
\end{enumerate}
\label{5singfiberscor}
\end{cor}

\begin{proof}
This simply involves testing that $p_1p_2 \cdots p_5 \leq \nicefrac{N_0}{2}$ in the given cases. Then as $\Delta(p_1p_2 \cdots p_5) = 2$ and $\Delta(N_0 - p_1p_2 \cdots p_5) = -2$ the proposition follows easily. 
\end{proof}

The remaining finite cases for five singular fibers can all be checked to satisfy $U\cdot HF_{red}(Y)$. Thus to prove Theorem \ref{5singfibersthm} we only need to address the following infinite cases: $\Sigma(2,3,5,p,q)$, $\Sigma(2,3,7,p,q)$, $\Sigma(2,3,11,13,p)$.\par

We will address each of these infinite cases separately in the following lemmas, proved later in this section. 

\begin{restatable}[]{lemma}{lemAapq}
Let $Y = \Sigma(2,3,5,p,q),$ for $p \neq 7$. Then $U \cdot HF_{red}(U) \neq 0.$
\label{235pqlem}
\end{restatable}

\begin{restatable}[]{lemma}{lemAbpq}
Let $Y = \Sigma(2,3,5,7,q).$ Then $U \cdot HF_{red}(U) \neq 0.$
\label{2357pqlem}
\end{restatable}

\begin{restatable}[]{lemma}{lemBpq}
Let $Y = \Sigma(2,3,7,p,q).$ Then $U \cdot HF_{red}(U) \neq 0.$
\label{237pqlem}
\end{restatable}

\begin{restatable}[]{lemma}{lemCpq}
Let $Y = \Sigma(2,3,11,13,q).$ Then $U \cdot HF_{red}(U) \neq 0.$
\label{other5lem}
\end{restatable}

With these lemmas, we are ready to prove Theorem \ref{5singfibersthm}

\singfibersthm*

\begin{proof}
This follows immediately from Corollary \ref{5singfiberscor} and Lemmas \ref{235pqlem}, \ref{237pqlem} and \ref{other5lem}. 
\end{proof}

\subsection{$\Sigma(2,3,5,p,q)$ Preliminaries}
Before we proceed with our proof of Lemmas \ref{235pqlem} and \ref{2357pqlem} we must first complete some algebraic preliminaries. Consider $Y = \Sigma(2,3,5,p,q),$ with Diophantine solutions $D = (e_0, 1, a, b, p', q')$. Then the Diophantine equation simplifies as follows 
$$-1 = 30pqe_0 + 15pq + 10apq + 6bpq + 30p'q + 30pq'. $$
Reducing this mod 30 we get the following,
\begin{align*}
29 &\equiv 15pq + 10apq + 6bpq  \mod 30 \\
29 &\equiv pq(15 + 10a +6b) \mod 30
\end{align*}

By the requirements of the Diophantine equation we know that $a \in \{1,2\}$ and $b \in \{1, 2, 3, 4\}$. Using this and reducing $pq \mod 30$ we can solve for $a$ and $b$. That gives us the following,

\begin{align*}
&pq \equiv 1 \mod 30 \implies D = (e_0,1,2,4,p',q') \quad \quad  &&pq \equiv 17 \mod 30 \implies D = (e_0,1,1,2,p',q') \\
&pq \equiv 7 \mod 30 \implies D = (e_0,1,2,2,p',q')\quad \quad &&pq \equiv 19 \mod 30 \implies D = (e_0,1,2,1,p',q') \\
&pq \equiv 11 \mod 30 \implies D = (e_0,1,1,4,p',q')\quad \quad &&pq \equiv 23 \mod 30 \implies D = (e_0,1,1,3,p',q') \\
&pq \equiv 13 \mod 30 \implies D = (e_0,1,2,3,p',q')\quad \quad &&pq \equiv 29 \mod 30 \implies D = (e_0,1,1,1,p',q') \\
\end{align*}

Note in all the above cases, $apq \equiv 2 \mod 3$ and $bpq \equiv 4 \mod 5$. \par

We can also come up with a bound on $e_0$ using the Diophantine equation:
\begin{align*}
-1 &= 30pqe_0 + 15pq + 10apq + 6bpq + 30p'q + 30pq' \\
-30pqe_0 &= 15pq + 10apq+6bpq+30p'q + 30pq' + 1 \\
|e_0| &= \frac{15pq + 10apq + 6bpq + 30p'q + 30pq' + 1}{30pq}\\
|e_0| &= \frac{1}{2} + \frac{a}{3} + \frac{b}{5} + \frac{p'}{p} + \frac{q'}{q} + \frac{1}{30pq} \\
|e_0| &< \frac{1}{2} + \frac{2}{3} + \frac{4}{5} + \frac{p}{p} + \frac{q}{q} + \frac{1}{30\cdot 7\cdot 11}= \frac{4582}{1155} \approx 3.96\\
\end{align*}
Therefore $e_0 = -1, -2, -3$. \par

As stated above, the goal is to find a consecutive sequence of values
$$(\Delta(x), \Delta(x+1), \dots, \Delta(x+j-1), \Delta(x+j) = (1, 0, \dots, 0, 1).$$
In the majority of cases for $\Sigma(2,3,5,p,q)$ this will be a pair of elements of the form $xpq, xpq \pm 1$. With this in mind, consider the following simplification of $\Delta(xpq + k)$, for $x \in \Z_+$ and $k \in \Z$.

\begin{align*}
    \Delta(xpq + k) &= 1 + |e_0|(xpq + k) - \left \lceil \frac{xpq+k}{2}\right \rceil- \left \lceil \frac{a(xpq+k)}{3}\right \rceil- \left \lceil \frac{b(xpq+k)}{5}\right \rceil \\
    &\quad \quad - \left \lceil \frac{p'(xpq+k)}{p}\right \rceil - \left \lceil \frac{q'(xpq+k)}{q}\right \rceil \\
    &= 1 + |e_0|xpq + k|e_0| - \left \lceil \frac{xpq+k}{2}\right \rceil - \left \lceil \frac{a(xpq+k)}{3}\right \rceil - \left \lceil \frac{b(xpq+k)}{5}\right \rceil \\
    &\quad \quad - xp'q - \left \lceil \frac{kp'}{p}\right \rceil - xpq' - \left \lceil \frac{kq'}{q}\right \rceil\\
    &= 1 + |e_0|xpq + k|e_0| - \left \lceil \frac{xpq+k}{2}\right \rceil - \left \lceil \frac{a(xpq+k)}{3}\right \rceil - \left \lceil \frac{b(xpq+k)}{5}\right \rceil \\
    &\quad \quad - \left \lceil \frac{kp'}{p}\right \rceil - \left \lceil \frac{kq'}{q}\right \rceil - x \left(-pqe_0 - \frac{pq}{2} - \frac{apq}{3} - \frac{bpq}{5} - \frac{1}{30} \right) \\
    &= 1 + k|e_0| + \frac{xpq}{2} - \left \lceil \frac{xpq+k}{2}\right \rceil + \frac{axpq}{3} - \left \lceil \frac{a(xpq+k)}{3}\right \rceil + \frac{bxpq}{5} - \left \lceil \frac{b(xpq+k)}{5}\right \rceil \\
    &\quad\quad- \left \lceil \frac{kp'}{p}\right \rceil- \left \lceil \frac{kq'}{q}\right \rceil + \frac{x}{30}.
\end{align*}

Using this simplification above we are able to compute $\Delta(xpq)$ and $\Delta(xpq \pm 1)$ for the $x$ values that satisfy the majority of our $\Sigma(2,3,5,p,q)$ cases. These values are presented in Table \ref{235pqfunctionvalues}. 
\begin{table}[h]
    \centering
    \begin{tabular}{|c||c|c|c|c|c|c|c|c|} 
        \hline 
         &$a = 1$ & $a = 1$ & $a = 1$ & $a = 1$ & $a = 2$ & $a = 2$ & $a = 2$ & $a = 2$ \\
         &$b = 1$ & $b = 2$ & $b = 3$ & $b = 4$ & $b = 1$ & $b = 2$ & $b = 3$ & $b = 4$ \\ \hline \hline 
         \rowcolor{lightgray!50}$\Delta(25pq - 1) =$  & $2 - |e_0|$ & $2 - |e_0|$ & $2 - |e_0|$ & $2 - |e_0|$&$3- |e_0|$ &$3- |e_0|$ &$3- |e_0|$ &$3- |e_0|$  \\ \hline
          \rowcolor{lightgray!50}$\Delta(25pq) =$  & 1 & 1 & 1 & 1 & 1 & 1 & 1 & 1  \\ \hline 
           \rowcolor{lightgray!50}$\Delta(25pq + 1) =$  & $|e_0| - 2$ & $|e_0| - 2$ & $|e_0| - 2$ & $|e_0| - 2$ & $|e_0| - 3$ &$ |e_0| - 3$ &$|e_0| - 3$ &$|e_0| - 3$  \\ \hline \hline 
           $\Delta(26pq - 1) =$  & $2 - |e_0|$ & $2 - |e_0|$ & $2 - |e_0|$ & $3 - |e_0|$&$2- |e_0|$ &$2- |e_0|$ &$2- |e_0|$ &$3- |e_0|$  \\ \hline
          $\Delta(26pq) =$  & 1 & 1 & 1 & 1 & 1 & 1 & 1 & 1  \\ \hline 
           $\Delta(26pq + 1) =$  & $|e_0| - 2$ & $|e_0| - 3$ & $|e_0| - 3$ & $|e_0| - 3$ & $|e_0| - 2$ &$ |e_0| - 3$ &$|e_0| - 3$ &$|e_0| - 3$  \\ \hline \hline 
           \rowcolor{lightgray!50}$\Delta(27pq - 1) =$  & $2 - |e_0|$ & $2 - |e_0|$ & $3 - |e_0|$ & $3 - |e_0|$&$2- |e_0|$ &$2- |e_0|$ &$3- |e_0|$ &$3- |e_0|$  \\ \hline
         \rowcolor{lightgray!50} $\Delta(27pq) =$  & 1 & 1 & 1 & 1 & 1 & 1 & 1 & 1  \\ \hline 
          \rowcolor{lightgray!50} $\Delta(27pq + 1) =$  & $|e_0| - 2$ & $|e_0| - 2$ & $|e_0| - 3$ & $|e_0| - 3$ & $|e_0| - 2$ &$ |e_0| - 2$ &$|e_0| - 3$ &$|e_0| - 3$  \\ \hline 
    \end{tabular}
    \caption{Key values of $\Delta(xpq)$ and $\Delta(xpq \pm 1)$}
    \label{235pqfunctionvalues}
\end{table}

%%%%%%%%%%%%%%%%%%%%%%%%%%%%%%%%%%%%%%%%%%%%%%%%%%%%%%%%%%%%%%%%%%%%%%%%%%%%%%%%%%%%%%%%%%%%%%%%%%%%%%%%%%%%%%%%%%%%%%%%%%%%%%%%%%%%%%%%%%%%%%%%%%%%%%%%%%%%%%%%%%%%%

\subsection{$\Sigma(2,3,5,p,q)$ Results}

Now we are ready to prove the Lemma \ref{235pqlem}.

\lemAapq*

As stated above, we wish to find values that provide us with 2 consecutive $+1's$ in our $\Delta$-sequence that occur prior to $\nicefrac{N_0}{2}$. Then by Lemma \ref{DeltaCond} our proof is complete. 
The values that satisfy this condition for sufficiently large $p$ and $q$ are found in Tables \ref{Compmod30}, \ref{1mod30}, \ref{13mod30}, \ref{17mod30} and $\ref{29mod30}.$ We separate the values these tables into two cases. Type A are those where the values are of the form $xpq \pm k$. Type B cases are those that follow a different pattern and are color coded green. \par

\begin{table}[h]
    \centering
    \begin{tabular}{|c||c|c|c|}
\hline
&$e_0 = -1$ & $e_1 = -2$ & $e_0 = -3$ \\ \hline \hline 
\multirow{2}{10em}{$pq \equiv 1 \mod 30$} & $26pq - 1$ & $26pq - 1$ & See Table \ref{1mod30}  \\
& $26pq$ & $26pq$ &   \\\hline 
\multirow{2}{10em}{$pq \equiv 7 \mod 30$} & $26pq - 1$ & $25pq - 1$ & $27pq$ \\
& $26pq$ & $25pq$ & $27pq+1$ \\\hline 
\multirow{2}{10em}{$pq \equiv 11 \mod 30$} & $25pq - 1$ & $26pq - 1$ & $25pq$ \\
 & $25pq$ & $26pq$ & $25pq+1$ \\\hline 
\multirow{2}{10em}{$pq \equiv 13 \mod 30$} & $26pq - 1$ & $25pq - 1$ & See Table \ref{13mod30}  \\
& $26pq$ & $25pq$ &   \\\hline 
\multirow{2}{10em}{$pq \equiv 17 \mod 30$} & $25pq - 1$ & See Table \ref{17mod30} & $25pq$ \\
& $25pq$ &  & $25pq+1$ \\\hline 
\multirow{2}{10em}{$pq \equiv 19 \mod 30$} & $26pq - 1$ & $25pq - 1$ & $ 26pq+1$ \\
& $26pq$ & $25pq$ & $26pq+1$ \\\hline 
\multirow{2}{10em}{$pq \equiv 23 \mod 30$} & $25pq - 1$ & $27pq - 1$ & $25pq$ \\
& $25pq$ & $27pq$ & $25pq+1$ \\\hline 
\multirow{2}{10em}{$pq \equiv 29 \mod 30$} & $25pq - 1$ & See Table \ref{29mod30} & $25pq$ \\
& $25pq$ &  & $25pq+1$ \\\hline 
\end{tabular}
    \caption{General Cases for $Y = \Sigma(2,3,5,p,q)$, with $p$ and $q$ sufficiently large}
    \label{Compmod30}
\end{table}

\begin{table}[h]
    \centering
    \begin{tabular}{|c||c|c|c|c|}
    \hline
         & $0<\nicefrac{p'}{p}<\nicefrac{1}{3}$ &$\nicefrac{1}{3}<\nicefrac{p'}{p}<\nicefrac{1}{2}$ &$\nicefrac{1}{2}<\nicefrac{p'}{p}<\nicefrac{2}{3}$
         &$\nicefrac{2}{3}<\nicefrac{p'}{p}<1$
         \\ \hline \hline
         \multirow{3}{7em}{$0<\nicefrac{q'}{q}<\nicefrac{1}{3}$} & $20pq$ & \cellcolor{green!25}$15pq + 30pq'$ & \cellcolor{green!25}$15pq + 30pq'$ & $20pq - 3$\\ 
         & $20pq+1$ & \cellcolor{green!25}$15pq + 30pq'+1$ & \cellcolor{green!25}$15pq + 30pq'+1$ & {\tiny$\vdots$} \\
         & $20pq+2$ & \cellcolor{green!25} & \cellcolor{green!25} & $22pq$\\\hline
         \multirow{3}{7em}{$\nicefrac{1}{3}<\nicefrac{q'}{q}<\nicefrac{1}{2}$} & \cellcolor{green!25} $15pq + 30p'q$ & $20pq$ & $22pq - 3$ & $22pq - 3$\\
         & \cellcolor{green!25} $15pq + 30p'q + 1$ & $20pq+1$ & {\tiny$\vdots$} & {\tiny$\vdots$}\\
         & \cellcolor{green!25} & $20pq + 2$ & $22pq$ & $22pq$\\\hline
         \multirow{3}{7em}{$\nicefrac{1}{2}<\nicefrac{q'}{q}<\nicefrac{2}{3}$} & \cellcolor{green!25} $15pq + 30p'q$ & $22pq - 3$ & $22pq - 2$ & $22pq - 3$\\
         & \cellcolor{green!25} $15pq + 30p'q + 1$ & {\tiny$\vdots$} & $22pq - 1$ & {\tiny$\vdots$}\\
         & \cellcolor{green!25} & $20pq + 2$ & $22pq$ & $22pq$\\\hline
         \multirow{3}{7em}{$\nicefrac{2}{3}<\nicefrac{q'}{q}<1$} & $22pq - 3$ & $22pq - 3$ & $22pq - 3$ & $22pq - 2$\\
         & {\tiny$\vdots$} & {\tiny$\vdots$} & {\tiny$\vdots$} & {\tiny$\vdots$}\\
         & $22pq$ & $22pq$ & $22pq$ & $22pq$\\\hline
    \end{tabular}
    \caption{$pq \equiv 1 \mod 30$, $e_0 = -3$, , for $p$ and $q$ sufficiently large}
    \label{1mod30}
\end{table}

\begin{table}[h]
    \centering
    \begin{tabular}{|c||c|c|c|c|c|}
    \hline
         & $0<\nicefrac{p'}{p}\leq \nicefrac{1}{5}$ & $\nicefrac{1}{5}<\nicefrac{p'}{p}\leq \nicefrac{2}{5}$ & $\nicefrac{2}{5}<\nicefrac{p'}{p}\leq \nicefrac{3}{5}$ & $\nicefrac{3}{5}<\nicefrac{p'}{p}\leq \nicefrac{4}{5}$ & $\nicefrac{4}{5}<\nicefrac{p'}{p}\leq 1$  \\ \hline\hline
        \multirow{3}{7em}{$0<\nicefrac{q'}{q}\leq \nicefrac{1}{5}$} & $26pq$ & $26pq$ & $26pq$ & $26pq$ & $26pq$\\
        & $26pq + 1$ & $26pq+1$ & $26pq + 1$ & $26pq + 1$ & $26pq + 1$\\
        & $26pq + 2$ & $26pq + 2$ & $26pq + 2$ & $26pq + 2$ & $26pq + 2$\\ \hline
        
         \multirow{3}{7em}{$\nicefrac{1}{5}<\nicefrac{q'}{q}\leq \nicefrac{2}{5}$} & $26pq$ & $26pq$ & $26pq$ & $26pq$ & $26pq$ \\
         & $26pq + 1$ & $26pq + 1$ & $26pq + 1$ & $26pq + 1$ & $26pq + 1$ \\
         & $26pq + 2$ & $26pq + 2$ & $26pq + 2$ & $26pq + 2$ & $26pq + 2$ \\\hline
         \multirow{3}{7em}{$\nicefrac{2}{5}<\nicefrac{q'}{q}\leq \nicefrac{3}{5}$} & $26pq$ & $26pq$ & $28pq$ & $28pq$ & $28pq - 4$\\
         & $26pq + 1$ & $26pq + 1$ & {\tiny$\vdots$} & {\tiny$\vdots$} & {\tiny$\vdots$}\\
         & $26pq + 2$ & $26pq + 2$ & $28pq + 6$ & $28pq + 6$ & $ 28pq$\\\hline
         \multirow{3}{7em}{$\nicefrac{3}{5}<\nicefrac{q'}{q}\leq \nicefrac{4}{5}$} & $26pq$ & $26pq$ & $28pq$ & \cellcolor{green!25}$90pq'-3pq$ & $28pq - 4$ \\
         & $26pq+1$ & $26pq+1$ & {\tiny$\vdots$} & \cellcolor{green!25}{\tiny$\vdots$} & {\tiny$\vdots$} \\
         & $26pq + 2$ & $26pq + 2$ & $28pq+6$ &\cellcolor{green!25} $90pq' - 3pq + 3$ & $28pq$ \\\hline
         \multirow{3}{7em}{$\nicefrac{4}{5}<\nicefrac{q'}{q}\leq 1$} & $26pq$ & $26pq$ & $28pq - 4$ & $28pq - 4$ & $28pq - 4$ \\
         & $26pq+1$ & $26pq+1$ & {\tiny$\vdots$} & {\tiny$\vdots$} & {\tiny$\vdots$} \\
         & $26pq + 2$ & $26pq + 2$ & $28pq$ & $28pq$ & $28pq$ \\
         
         \hline
    \end{tabular}
    \caption{$pq \equiv 13 \mod 30$, $e_0 = -3$, for $p$ and $q$ sufficiently large}
    \label{13mod30}
\end{table}

\begin{table}[h]
    \centering
    \begin{tabular}{|c||c|c|c|c|c|}
    \hline
         & $0<\nicefrac{p'}{p}\leq \nicefrac{1}{5}$ & $\nicefrac{1}{5}<\nicefrac{p'}{p}\leq \nicefrac{2}{5}$ & $\nicefrac{2}{5}<\nicefrac{p'}{p}\leq \nicefrac{3}{5}$ & $\nicefrac{3}{5}<\nicefrac{p'}{p}\leq \nicefrac{4}{5}$ & $\nicefrac{4}{5}<\nicefrac{p'}{p}\leq 1$  \\ \hline\hline
         \multirow{3}{7em}{$0<\nicefrac{q'}{q}\leq \nicefrac{1}{5}$} & $28pq$ & $28pq$ & $28pq$ & $26pq -2$ & $26pq-2$\\
         & {\tiny$\vdots$} & {\tiny$\vdots$} & {\tiny$\vdots$} & $26pq-1$ & $26pq-1$\\
         & $28pq + 3$ & $28pq + 4$ & $28pq + 4$ & $26pq$ & $ 26pq$\\\hline
         \multirow{3}{7em}{$\nicefrac{1}{5}<\nicefrac{q'}{q}\leq \nicefrac{2}{5}$} & $28pq$ & \cellcolor{green!25}$51pq - 60pq'-3$ & $28pq - 6$ & $26pq - 2$ & $26pq - 2$ \\
         & {\tiny$\vdots$} & \cellcolor{green!25}{\tiny$\vdots$} & {\tiny$\vdots$} & $26pq - 1$ & $26pq - 1$ \\
         & $28pq + 4$ & \cellcolor{green!25}$51pq - 60pq'$ & $28pq$ & $26pq$ & $26pq$ \\\hline
         \multirow{3}{7em}{$\nicefrac{2}{5}<\nicefrac{q'}{q}\leq \nicefrac{3}{5}$} & $28pq$ & $28pq - 6$ & $28pq - 6$ & $26pq - 2$ & $26pq - 2$\\
         & {\tiny$\vdots$} & {\tiny$\vdots$} & {\tiny$\vdots$} & $26pq - 1$ & $26pq - 1$\\
         & $28pq$ & $28pq$ & $28pq$ & $26pq$ & $26pq$\\\hline
         \multirow{3}{7em}{$\nicefrac{3}{5}<\nicefrac{q'}{q}\leq \nicefrac{4}{5}$} & $26pq - 2$ & $26pq -2$ & $26pq -2$ & $26pq -2$ & $26pq -2$ \\
         & $26pq - 1$ & $26pq -1$ & $26pq -1$ & $26pq -1$ & $26pq -1$ \\
         & $26pq$ & $26pq$ & $26pq$ & $26pq$ & $26pq$ \\\hline
         \multirow{3}{7em}{$\nicefrac{5}{5}<\nicefrac{q'}{q}\leq 1$} & $26pq - 2$ & $26pq -2$ & $26pq -2$ & $26pq -2$ & $26pq -2$ \\
         & $26pq - 1$ & $26pq -1$ & $26pq -1$ & $26pq -1$ & $26pq -1$ \\
         & $26pq$ & $26pq$ & $26pq$ & $26pq$ & $26pq$ \\\hline
    \end{tabular}
    \caption{$pq \equiv 17 \mod 30$, $e_0 = -2$, for $p$ and $q$ sufficiently large}
    \label{17mod30}
\end{table}

\begin{table}[h]
    \centering
    \begin{tabular}{|c||c|c|c|}
    \hline
         & $0<\nicefrac{p'}{p}<\nicefrac{1}{2}$ &$\nicefrac{1}{2}<\nicefrac{p'}{p}<\nicefrac{3}{5}$ &$\nicefrac{3}{5}<\nicefrac{p'}{p}<1$  \\ \hline \hline
         \multirow{3}{7em}{$0<\nicefrac{q'}{q}<\nicefrac{1}{2}$} & $26pq$ & \cellcolor{green!25}$60p'q - 8pq$ & \cellcolor{green!25}$45pq - 30p'q - 1$\\
         & $26pq+1$ & \cellcolor{green!25}$60p'q - 8pq + 1$& \cellcolor{green!25}$45pq - 30p'q$\\
         & $26pq+2$ &\cellcolor{green!25} $60q' - 8q + 2$ &\cellcolor{green!25} \\\hline
         \multirow{3}{7em}{$\nicefrac{1}{2}<\nicefrac{q'}{q}<\nicefrac{3}{5}$} & \cellcolor{green!25}$60pq' - 8pq$ & $26pq - 2$ &$26pq - 2$ \\
         & \cellcolor{green!25}$60pq' - 8pq + 1$ & $26pq - 1$ &$26pq - 1$ \\
         & \cellcolor{green!25}$60pq' - 8pq+2$ & $26pq$ &$26pq $ \\\hline
         \multirow{3}{7em}{$\nicefrac{3}{5}<\nicefrac{q'}{q}<1$} & \cellcolor{green!25}$45pq - 30pq' - 1$ & $26pq - 2$ & $26pq - 2$\\
        & \cellcolor{green!25}$45pq - 30pq'$ & $26pq - 1$ & $26pq - 1$ \\
        & \cellcolor{green!25} & $26pq$ & $26pq$ \\\hline
    \end{tabular}
    \caption{$pq \equiv 29 \mod 30$, $e_0 = -2$, for $p$ and $q$ sufficiently large}
    \label{29mod30}
\end{table}

\begin{proof}
Given the values in Tables \ref{Compmod30}, \ref{1mod30}, \ref{13mod30}, \ref{17mod30} and $\ref{29mod30},$ our proof involves checking the value of $\Delta$ of these numbers, and confirming that they occur before $\nicefrac{N_0}{2}.$ For Type A cases, the $\Delta$ function of these numbers can either be found in Table \ref{235pqfunctionvalues} or computed using the formula for $\Delta(xpq + k)$. Type B cases we will address later. \par

For the Type A cases it remains to show that these values occur before $\nicefrac{N_0}{2}$. Since it is clear that $$ 25pq \leq 26pq \leq 27pq \leq 28pq \leq 28pq + 6$$
we need only to check that $28pq + 6 \leq \nicefrac{N_0}{2}.$ It can be easily checked that this holds when $$p>10, \text{ and } q \geq \frac{10p + 4}{p - 10},$$ which are satisfied in the following cases,
\begin{itemize}
    \item $p = 11,\quad q>114$
    \item $p = 13, \quad q > 43$
    \item $p = 17,\quad q>24$
    \item $p = 19,\quad q>21$
    \item $p > 20$
\end{itemize}
The finite cases that do not satisfy these inequalities may be easily checked to satisfy $$U \cdot HF_{red} \neq 0.$$ 

Thus it remains to address the following. 
\begin{enumerate}
    \item $pq \equiv 1 \mod 30$, $e_0 = -3$, $\nicefrac{1}{3} < \nicefrac{p'}{p} < \nicefrac{2}{3}$ and $0 < \nicefrac{q'}{q} < \nicefrac{1}{3}$, and vice versa. 
    \item $pq \equiv 13 \mod 30$, $e_0 = -3$, $\nicefrac{3}{5} < \nicefrac{p'}{p} < \nicefrac{4}{5}$ and $\nicefrac{3}{5} < \nicefrac{q'}{q} < \nicefrac{4}{5}$
    \item $pq \equiv 17\mod 30$, $e_0 = -2$, $\nicefrac{1}{5} < \nicefrac{p'}{p}\leq \nicefrac{2}{5}$ and $\nicefrac{1}{5} < \nicefrac{q'}{q}\leq \nicefrac{2}{5}$
    \item $pq \equiv 29 \mod 30$, $e_0 = -2$, $\nicefrac{1}{2} < \nicefrac{q'}{q} < \nicefrac{3}{5}$ and $0 < \nicefrac{p'}{p} < \nicefrac{1}{2}$ and vice versa
    \item $pq \equiv 29 \mod 30$, $e_0 = -2$, $\nicefrac{3}{5} < \nicefrac{p'}{p} < 1$ and $0 < \nicefrac{q'}{q} < \nicefrac{1}{2}$, and vice versa
\end{enumerate} 

\paragraph*{Case 1:} Assume that $\nicefrac{1}{3} < \nicefrac{p'}{p} < \nicefrac{2}{3}$ and $0 < \nicefrac{q'}{q} < \nicefrac{1}{3}$. First we check that $\Delta(15pq + 30pq') = 1$ and $\Delta(15pq + 30pq' + 1) = 1$. 

\begin{align*}
    \Delta(15pq + 30pq') &= 1 + |e_0|(15pq + 30pq') - \left \lceil \frac{15pq + 30pq'}{2} \right \rceil - \left \lceil \frac{2(15pq + 30pq')}{3} \right \rceil\\
    &\quad\quad - \left \lceil \frac{4(15pq + 30pq')}{5}\right \rceil - \left \lceil \frac{p'(15pq + 30pq'}{p} \right \rceil - \left \lceil \frac{q'(15pq + 30pq'}{q}\right \rceil\\
    &= 1 + 15pq|e_0| - 30pq'|e_0| - \frac{15pq + 30pq' + 1}{2} - \frac{30pq + 60pq'}{3} \\
    &\quad\quad - \frac{60pq + 120pq'}{5} - 15p'q - 30p'q' - 15pq' - \left \lceil \frac{30q'}{q}(pq')\right \rceil \\
    &= 1 + 15pq|e_0| + 30pq'|e_0| - \frac{15pq + 30pq' + 1}{2} - \frac{30pq + 60pq'}{3}\\
    &\quad\quad - \frac{60pq - 120pq'}{5} - 30p'q' - 15\left( -pqe_0 - \frac{pq}{2} - \frac{2pq}{3} - \frac{4pq}{5} - \frac{1}{30} \right) \\
    &\quad\quad - \left \lceil \frac{30q'}{q} \left( -pqe_0 - \frac{pq}{2} - \frac{2pq}{3} - \frac{4pq}{5} - \frac{1}{30} - p'q \right) \right \rceil \\
    &= 1
\end{align*}

It can be similarly checked that $\Delta(15pq + 30pq' + 1) = 1$. Now we must check that $$15pq + 30pq' + 1 < \nicefrac{N_0}{2}.$$ It is easily checked that this is satisfied when $$q' < \frac{29pq - 30q - 3p - 2}{60p}.$$
Since $q' < \nicefrac{q}{3}$ by assumption, we check when
$$\frac{q}{3} < \frac{29pq - 30q - 3p - 2}{60p},$$
which is satisfied when $q > 6.$ Thus $15pq + 30pq' + 1 < \nicefrac{N_0}{2}$ is satisfied for $q>6$, which is trivially true. The case where $\nicefrac{1}{3} < \nicefrac{q'}{q} < \nicefrac{2}{3}$ and $0 < \nicefrac{p'}{p} < \nicefrac{1}{3}$ can be checked the same way, and so this case is complete. \par 

The remaining cases can be checked using the same process.
   \paragraph*{Case 2:} Let $pq \equiv 13 \mod 30$, $e_0 = -3$, $\nicefrac{3}{5} < \nicefrac{p'}{p},\nicefrac{q'}{q} <\nicefrac{4}{5}$. It is easily checked that $$\Delta(60pq' - 9pq) = 1 = \Delta(60pq' - 9pq + 3).$$ $$\Delta(60pq' - 9pq + 1) = 0 = \Delta(60pq' - 9pq + 2).$$
    Note that if $\left \lceil \frac{3p'}{p}\right\rceil = 3$, then $q>\frac{3}{5}$, by the Diophantine equation. Since we assumed $\nicefrac{3}{5} < \nicefrac{q'}{q}<\nicefrac{4}{5}$, we must have $\left \lceil \frac{3p'}{p}\right\rceil = 2$. Also using the Diophantine equation we have that $\nicefrac{p'}{p} > \nicefrac{3}{5}$, implies $\nicefrac{q'}{q}< \nicefrac{19}{30}$.  We check that $$60pq' - 9pq + 3 < \nicefrac{N_0}{2},$$ when
    $$q'<\frac{77pq - 30p - 30q - 6}{120p}.$$
    Since $q' < \frac{19q}{30}$, we check when 
    $$\frac{19q}{30} < \frac{77pq - 30p - 30q - 6}{120p},$$
    which is satisfied when $$p>30, \text{ and } q > \frac{30p + 6}{p - 30}.$$ The above inequalities are satisfied in the following cases,
    \begin{multicols}{2}
    \begin{itemize}
        \item $p> 60$
        \item $p = 31, \quad q>936$
        \item $p = 37, \quad q>159$
        \item $p = 41, \quad q>112$
        \item $p = 43, \quad q>99$
        \item $p = 47, \quad q>83$
        \item $p = 53, \quad q>69$
        \item $p = 59, \quad q>61$
    \end{itemize}
    \end{multicols}
    We still must address this for $p< 30$. We check each of these cases individually, using our restrictions on $\nicefrac{p'}{p}$, $pq \mod 30$ and the Diophantine equation. In each case we are trying to show that 
    \begin{equation}
    q' < \frac{77pq - 30q - 30p - 3}{130p}.
    \label{Necineq}
    \end{equation}
    \begin{itemize}
        \item $p = 11$: Then $p' \in \{7,8\}$. By the Diophantine equation we have $q' \leq \frac{197q - 1}{330}$, and this satisfies \eqref{Necineq} when $q > \frac{329}{11}$. 
        \item $p = 13$: Then $p' \in \{8, 9, 10\}$. By the Diophantine equation we have $q' \leq \frac{241q - 1}{390}$, and this satisfies \eqref{Necineq} when $q > 11$. 
        \item $p = 17$: Then $p' \in \{11, 12, 13\}$. By the Diophantine equation we have $q' \leq \frac{299q - 1}{510}$, and this satisfies \eqref{Necineq} when $q > 55$. 
        \item $p = 19$: Then $p'\in \{12, 13, 14, 15\}$. By the Diophantine equation we have $q' \leq \frac{343q - 1}{570}$, and this satisfies \eqref{Necineq} when $q > 9$. 
        \item $p = 23$: Then $p' \in \{14, 15, 16, 17, 18\}$. By the Diophantine equation we have $q' \leq \frac{431q - 1}{690}$, and this satisfies \eqref{Necineq} when $q > 40$. 
        \item $p = 29$: Then $p' \in \{18, 19, 20, 21, 22, 23\}$. By the Diophantine equation we have $q' \leq \frac{533q - 1}{870}$, and this satisfies \eqref{Necineq} when $q > 12$. 
    \end{itemize}
    The remaining finite cases can be checked by hand, and thus this case is complete.

  \paragraph*{Case 3:} $pq \equiv 17 \mod 30$, $e_0 = -2$, $\nicefrac{1}{5} < \nicefrac{p'}{p},\nicefrac{q'}{q} <\nicefrac{2}{5}$. \\ It is easily checked that $$\Delta(51pq - 60pq' - 3) = 1 = \Delta(51pq - 60pq').$$ $$\Delta(51pq - 60pq' - 2) = 0 = \Delta(51pq - 60pq' - 1).$$
    Using our assumptions on $p'$ and $q'$ and the Diophantine equation we have that $$51pq - 60pq' < \nicefrac{N_0}{2},$$ when
    $p>30, \text{ and } q> \frac{30p}{p-30}.$
    The above inequalities are satisfied in the following cases,
    \begin{multicols}{2}
    \begin{itemize}
        \item $p > 60$
        \item $p = 31, \quad q > 930$
        \item $p = 37, \quad q > 158$
        \item $p = 41, \quad q > 111$
        \item $p = 43, \quad q > 99$
        \item $p = 47, \quad q > 82$
        \item $p = 53, \quad q > 69$
        \item $p = 59, \quad q > 61$
    \end{itemize}
    \end{multicols}
    We still must address this for $p< 30$. We check each of these cases individually, using our restrictions on $\nicefrac{p'}{p}$, $pq \mod 30$ and the Diophantine equation. In each case we are trying to show that 
    \begin{equation}
    q' > \frac{43pq + 30p + 30q}{120p}.
    \label{Necineq2}
    \end{equation}
    \begin{itemize}
        \item $p = 11$: Then $p' \in \{3,4\}$. By the Diophantine equation we have $q' \geq \frac{133q - 1}{210}$, and this satisfies \eqref{Necineq2} when $q > 1$. 
        \item $p = 13$: Then $p' \in \{3, 4, 5\}$. By the Diophantine equation we have $q' \geq \frac{149q - 1}{390}$, and this satisfies \eqref{Necineq} when $q > 1$. 
        \item $p = 17$: Then $p' \in \{4, 5, 6\}$. By the Diophantine equation we have $q' \geq \frac{211q - 1}{510}$, and this satisfies \eqref{Necineq2} when $q > 6$. 
        \item $p = 19$: Then $p'\in \{12, 13, 14, 15\}$. By the Diophantine equation we have $q' \geq \frac{227q - 1}{570}$, and this satisfies \eqref{Necineq2} when $q > 9$. 
        \item $p = 23$: Then $p' \in \{14, 15, 16, 17, 18\}$. By the Diophantine equation we have $q' \geq \frac{259q - 1}{690}$, and this satisfies \eqref{Necineq2} when $q > 40$. 
        \item $p = 29$: Then $p' \in \{18, 19, 20, 21, 22, 23\}$. By the Diophantine equation we have $q' \geq \frac{337q - 1}{870}$, and this satisfies \eqref{Necineq2} when $q > 12$. 
    \end{itemize}
    The remaining finite cases can be checked by hand, and thus this case is complete.
    
  \paragraph*{Case 4:} $pq \equiv 29 \mod 30$, $e_0 = -2$, $\nicefrac{1}{2} < \nicefrac{q'}{q}<\nicefrac{3}{5}$ and $0 < \nicefrac{p'}{p} < \nicefrac{1}{2}$.\\ It is easily checked that $$\Delta(60pq' - 8pq) = 1 = \Delta(60pq' - 8pq + 2),$$
    $$\Delta(60pq' - 8pq + 1) = 0.$$
    Using our assumptions of $p'$ and $q'$ and the Diophantine equation we have that  $$60pq' -8pq + 2 < \nicefrac{N_0}{2},$$ when
    $p>10$ and $q > \frac{30p + 4}{3p - 30}$. These are satisfied in the following cases,
    \begin{itemize}
        \item $p>20$
        \item $p = 11, \quad q>111$
        \item $p = 13, \quad q>43$
        \item $p = 17, \quad q>25$
        \item $p = 19, \quad q>21$
    \end{itemize}
    The case where $\nicefrac{1}{2} < \nicefrac{p'}{p}<\nicefrac{3}{5}$ and $0 < \nicefrac{q'}{q} < \nicefrac{1}{2}$ can be checked similarly, completing this case.

    \paragraph*{Case 5:} $pq \equiv 29 \mod 30$, $e_0 = -2,$ $0 < \nicefrac{q'}{q}< \nicefrac{1}{2},$ $\nicefrac{3}{5} < \nicefrac{p'}{p}<1$. \\
    It is easily checked that $$\Delta(45pq - 30p'q - 1) = 1 = \Delta(45pq - 30p'q).$$
    Using our assumptions on $p'$ and $q'$ and the Diophantine equation we have that $$45pq - 30p'q < \nicefrac{N_0}{2},$$
    when $p>6$ and $q > \frac{6p}{p-6}$. These inequalities are satisfied in the following cases,
    \begin{itemize}
        \item $p>12$
        \item $p = 11, \quad q> 13$
    \end{itemize}
The case where $0 < \nicefrac{p'}{p}< \nicefrac{1}{2},$ $\nicefrac{3}{5} < \nicefrac{q'}{q}<1$, can be checked similarly. Thus our proof is complete.
\end{proof}

Now we prove the other case of $\Sigma(2,3,5,p,q).$

\lemAbpq*

\begin{table}[h]
    \centering
    \begin{tabular}{|c||c|c|c|c|c|}
    \hline
         & $0<\nicefrac{p'}{p}\leq \nicefrac{1}{5}$ & $\nicefrac{1}{5}<\nicefrac{p'}{p}\leq \nicefrac{2}{5}$ & $\nicefrac{2}{5}<\nicefrac{p'}{p}\leq \nicefrac{3}{5}$ & $\nicefrac{3}{5}<\nicefrac{p'}{p}\leq \nicefrac{4}{5}$ & $\nicefrac{4}{5}<\nicefrac{p'}{p}\leq 1$  \\ \hline\hline
        \multirow{3}{7em}{$0<\nicefrac{q'}{q}\leq \nicefrac{1}{5}$} & $26pq$ & $26pq$ & $26pq$ & $26pq$ & $26pq$\\
        & $26pq + 1$ & $26pq+1$ & $26pq + 1$ & $26pq + 1$ & $26pq + 1$\\
        & $26pq + 2$ & $26pq + 2$ & $26pq + 2$ & $26pq + 2$ & $26pq + 2$\\ \hline
        
         \multirow{3}{7em}{$\nicefrac{1}{5}<\nicefrac{q'}{q}\leq \nicefrac{2}{5}$} & $26pq$ & $26pq$ & $26pq$ & $26pq$ & $26pq$ \\
         & $26pq + 1$ & $26pq + 1$ & $26pq + 1$ & $26pq + 1$ & $26pq + 1$ \\
         & $26pq + 2$ & $26pq + 2$ & $26pq + 2$ & $26pq + 2$ & $26pq + 2$ \\\hline
         \multirow{3}{7em}{$\nicefrac{2}{5}<\nicefrac{q'}{q}\leq \nicefrac{3}{5}$} & $26pq$ & $26pq$ & \cellcolor{red!25} & \cellcolor{green!25}$280q - 210q' - 7$ &\cellcolor{red!25} \\
         & $26pq + 1$ & $26pq + 1$ & \cellcolor{red!25} & \cellcolor{green!25}{\tiny$\vdots$} &\cellcolor{red!25} \\
         & $26pq + 2$ & $26pq + 2$ & \cellcolor{red!25} & \cellcolor{green!25}$280q - 210q'$ &\cellcolor{red!25} \\\hline
         \multirow{3}{7em}{$\nicefrac{3}{5}<\nicefrac{q'}{q}\leq \nicefrac{4}{5}$} & $26pq$ & $26pq$ & \cellcolor{green!25}$1420q - 1890q' - 7$ & \cellcolor{green!25}$60pq'-9pq$ & \cellcolor{red!25} \\
         & $26pq+1$ & $26pq+1$ & \cellcolor{green!25}{\tiny$\vdots$} & \cellcolor{green!25}{\tiny$\vdots$} & \cellcolor{red!25} \\
         & $26pq + 2$ & $26pq + 2$ & \cellcolor{green!25}$1420q - 1890q'$ &\cellcolor{green!25} $60pq' - 9pq + 3$ & \cellcolor{red!25} \\\hline
         \multirow{3}{7em}{$\nicefrac{4}{5}<\nicefrac{q'}{q}\leq 1$} & $26pq$ & $26pq$ & \cellcolor{green!25}$520q - 420q' - 2$ & \cellcolor{red!25} & \cellcolor{red!25} \\
         & $26pq+1$ & $26pq+1$ & \cellcolor{green!25}{\tiny$\vdots$} & \cellcolor{red!25} & \cellcolor{red!25} \\
         & $26pq + 2$ & $26pq + 2$ & \cellcolor{green!25}$520q - 420q'$ & \cellcolor{red!25} & \cellcolor{red!25} \\
         
         \hline
    \end{tabular}
    \caption{$pq \equiv 13 \mod 30$, $p = 7$, $e_0 = -3$}
    \label{13mod30p7}
\end{table}

\begin{table}[h]
    \centering
    \begin{tabular}{|c||c|c|c|c|c|}
    \hline
         & $0<\nicefrac{p'}{p}\leq \nicefrac{1}{5}$ & $\nicefrac{1}{5}<\nicefrac{p'}{p}\leq \nicefrac{2}{5}$ & $\nicefrac{2}{5}<\nicefrac{p'}{p}\leq \nicefrac{3}{5}$ & $\nicefrac{3}{5}<\nicefrac{p'}{p}\leq \nicefrac{4}{5}$ & $\nicefrac{4}{5}<\nicefrac{p'}{p}\leq 1$  \\ \hline\hline
         \multirow{3}{7em}{$0<\nicefrac{q'}{q}\leq \nicefrac{1}{5}$} & \cellcolor{red!25} & \cellcolor{red!25} & \cellcolor{green!25}$182q-2$ & $26pq -2$ & $26pq-2$\\
         & \cellcolor{red!25} & \cellcolor{red!25} & \cellcolor{green!25}{\tiny$\vdots$} & $26pq-1$ & $26pq-1$\\
         & \cellcolor{red!25} & \cellcolor{red!25} & \cellcolor{green!25}$182q$ & $26pq$ & $ 26pq$\\\hline
         \multirow{3}{7em}{$\nicefrac{1}{5}<\nicefrac{q'}{q}\leq \nicefrac{2}{5}$} & \cellcolor{red!25} & \cellcolor{green!25}$51pq - 60pq'-3$ & \cellcolor{green!25}$240q - 210q' - 10$ & $26pq - 2$ & $26pq - 2$ \\
         & \cellcolor{red!25} & \cellcolor{green!25}{\tiny$\vdots$} & \cellcolor{green!25}{\tiny$\vdots$} & $26pq - 1$ & $26pq - 1$ \\
         & \cellcolor{red!25} & \cellcolor{green!25}$51pq - 60pq'$ & \cellcolor{green!25}$240q - 210q' - 10$ & $26pq$ & $26pq$ \\\hline
         \multirow{3}{7em}{$\nicefrac{2}{5}<\nicefrac{q'}{q}\leq \nicefrac{3}{5}$} & \cellcolor{red!25}& \cellcolor{green!25}$777q - 1260q' - 7$ & \cellcolor{red!25} & $26pq - 2$ & $26pq - 2$\\
         & \cellcolor{red!25} & \cellcolor{green!25}{\tiny$\vdots$} & \cellcolor{red!25} & $26pq - 1$ & $26pq - 1$\\
         & \cellcolor{red!25} & \cellcolor{green!25}$777q - 1260q'$ & \cellcolor{red!25} & $26pq$ & $26pq$\\\hline
         \multirow{3}{7em}{$\nicefrac{3}{5}<\nicefrac{q'}{q}\leq \nicefrac{4}{5}$} & $26pq - 2$ & $26pq -2$ & $26pq -2$ & $26pq -2$ & $26pq -2$ \\
         & $26pq - 1$ & $26pq -1$ & $26pq -1$ & $26pq -1$ & $26pq -1$ \\
         & $26pq$ & $26pq$ & $26pq$ & $26pq$ & $26pq$ \\\hline
         \multirow{3}{7em}{$\nicefrac{5}{5}<\nicefrac{q'}{q}\leq 1$} & $26pq - 2$ & $26pq -2$ & $26pq -2$ & $26pq -2$ & $26pq -2$ \\
         & $26pq - 1$ & $26pq -1$ & $26pq -1$ & $26pq -1$ & $26pq -1$ \\
         & $26pq$ & $26pq$ & $26pq$ & $26pq$ & $26pq$ \\\hline
    \end{tabular}
    \caption{$pq \equiv 17 \mod 30$, $e_0 = -2$, $p=7$}
    \label{17mod30p7}
\end{table}

\begin{proof}
Consider Table \ref{Compmod30}. The largest value in this table is $27pq$, it is easily checked that for $p = 7$, $27pq < \nicefrac{N_0}{2}$ holds when $q > 42$. The remaining finite cases can be easily checked and then we have that $Y= \Sigma(2,3,5,7,q)$ satisfies $U\cdot HF_{red}(Y)$ except for the following cases,
\begin{enumerate}
    \item $pq \equiv 1 \mod 30$, $e_0 = -3$,
    \item $pq \equiv 29 \mod 30$, $e_0 = -2$,
    \item $pq \equiv 13 \mod 30$, $e_0 = -3$,
    \item $pq \equiv 17 \mod 30$, $e_0 = -2$.
\end{enumerate}

We will now address each of these cases separately. 

\paragraph{Case 1:} Consider Table \ref{1mod30}. Excluding Type B cases, the largest value here is $22pq$, and this satisfies $22pq < \nicefrac{N_0}{2}$ for $p=7$ when $q > \frac{14}{5}.$ Now it remains to check the Type B cases, but we have already shown that $15pq + 30pq' + 1 < \nicefrac{N_0}{2}$ for $q>6$. Thus this case is complete. 

\paragraph{Case 2:} Consider Table \ref{29mod30}, the largest value here is $26pq+2$, excluding Type B cases. It can be checked that for $p = 7$, $26pq+2 < \nicefrac{N_0}{2}$ is satisfied when $q > \nicefrac{214}{19}$. Now it remains to check the Type B cases. Consider when $\nicefrac{q'}{q} < \nicefrac{1}{2}$ and $\nicefrac{3}{5} < \nicefrac{p'}{p}$. We have already shown that $45pq - 30p'q < \nicefrac{N_0}{2}$ when $p>6$ and $q>\frac{6p}{p-6}$. Thus this is satisfied when $q > 42$ for $p=7$. The case where $\nicefrac{p'}{p} < \nicefrac{1}{2}$ and $\nicefrac{3}{5} < \nicefrac{q'}{q}$ can be checked similarly. Now consider when $\nicefrac{1}{2} < \nicefrac{q'}{q} < \nicefrac{3}{5}$ and $\nicefrac{p'}{p} < \nicefrac{1}{2}$. Since $p=7$, we must have $p' \in \{1,2,3\}$ to satisfy $\nicefrac{p'}{p} < \nicefrac{1}{2}$. When $p' = 1$ we can solve the Diophantine equation to find that 
$$q' = \frac{173q - 1}{210}$$
but then $\nicefrac{q'}{q} < \nicefrac{173}{210}$ and thus $\nicefrac{q'}{q} > \nicefrac{3}{5}$, so this case cannot occur for $p = 7$ and $p' = 1$.  For $p' = 2$ we follow the same argument to find that $$q' = \frac{143q - 1}{210},$$ but again this only satisfies $\nicefrac{1}{2} < \nicefrac{q'}{q} < \nicefrac{3}{5}$ for $\nicefrac{1}{38} < q < \nicefrac{1}{17}$. Therefore this case cannot occur for $p=7$ and $p' = 2$. For $p' = 3$ we check that $60pq' - 8pq + 2 < \nicefrac{N_0}{2}$ is satisfied for $q > \frac{210}{43}$ and thus this case is complete. The case where $\nicefrac{q'}{q} < \nicefrac{1}{2}, \nicefrac{1}{2} < \nicefrac{p'}{p} < \nicefrac{3}{5}$ can be checked to satisfy $60p'q - 8pq + 2 < \nicefrac{N_0}{2}$ when $q > \nicefrac{214}{15}$. Thus this case is complete. 

\paragraph{Case 3:} We proceed as in Case 2 and use the Diophantine equation to limit the possible values for $p'$ in each case. Then we check to make sure that $\nicefrac{q'}{q}$ satisfies the desired condition; if not that cell in the table is left blank and highlighted in red, as this case is not possible for $p = 7$. For the cases that are possible, we have found values that give us the desired 2 consecutive positive ones in our $\Delta$-sequence, presented in Table \ref{13mod30p7}. It remains to check that these occur before $\nicefrac{N_0}{2}$ in our $\Delta$ - sequence:
\begin{itemize}
    \item $\nicefrac{3}{5}<\nicefrac{p'}{p}<\nicefrac{4}{5}, \nicefrac{2}{5}<\nicefrac{q'}{q}<\nicefrac{3}{5}$\\
    It follows that $p' = 5,$ and thus $280q - 210q' < \nicefrac{N_0}{2}$ for $q > \frac{-208}{41}$.
    \item $\nicefrac{2}{5}<\nicefrac{p'}{p}<\nicefrac{3}{5}, \nicefrac{3}{5}<\nicefrac{q'}{q}<\nicefrac{4}{5}$\\
    It follows that $p' = 4,$ and thus $1420q - 1890q' - 7 < \nicefrac{N_0}{2}$ for $q > \frac{214}{45}$.
    \item $\nicefrac{3}{5}<\nicefrac{p'}{p}<\nicefrac{4}{5}, \nicefrac{3}{5}<\nicefrac{q'}{q}<\nicefrac{4}{5}$\\
    It follows that $p' = 5,$ and thus $60pq' - 9pq +3 < \nicefrac{N_0}{2}$ for $q > \frac{209}{73}$.
    \item $\nicefrac{2}{5}<\nicefrac{p'}{p}<\nicefrac{3}{5}, \nicefrac{4}{5}<\nicefrac{q'}{q}< 1$\\
    It follows that $p' = 3,$ and thus $520q - 420q' < \nicefrac{N_0}{2}$ for $q > \frac{214}{19}$.
\end{itemize}

\paragraph{Case 4:} We proceed exactly as in Case 3 and the values that give us the desired 2 consecutive positive ones in our $\Delta$-sequence are presented in Table \ref{17mod30p7}. It remains to check that these occur before $\nicefrac{N_0}{2}$ in our $\Delta$ - sequence:

\begin{itemize}
    \item $\nicefrac{2}{5}<\nicefrac{p'}{p}<\nicefrac{3}{5}, 0<\nicefrac{q'}{q}<\nicefrac{1}{5}$\\
    It is easily checked that $182q < \nicefrac{N_0}{2}$ for $q > \frac{210}{19}$.
    \item $\nicefrac{1}{5}<\nicefrac{p'}{p}<\nicefrac{2}{5}, \nicefrac{1}{5}<\nicefrac{q'}{q}<\nicefrac{2}{5}$\\
    It follows that $p' = 2,$ and thus $51pq - 60pq' < \nicefrac{N_0}{2}$ for $q > \frac{2354}{2331}$.
    
    \item $\nicefrac{2}{5}<\nicefrac{p'}{p}<\nicefrac{3}{5}, \nicefrac{1}{5}<\nicefrac{q'}{q}<\nicefrac{2}{5}$\\
    It follows that $p' = 3,$ and thus $240q - 210q' - 10 < \nicefrac{N_0}{2}$ for $q > \frac{64}{15}$.
    \item $\nicefrac{1}{5}<\nicefrac{p'}{p}<\nicefrac{2}{5}, \nicefrac{2}{5}<\nicefrac{q'}{q}< \nicefrac{3}{5}$\\
    It follows that $p' = 2,$ and thus $777q - 1260q' < \nicefrac{N_0}{2}$ for $q > \frac{222}{41}$.
\end{itemize}

\end{proof}

%%%%%%%%%%%%%%%%%%%%%%%%%%%%%%%%%%%%%%%%%%%%%%%%%%%%%%%%%%%%%%%%%%%%%%%%%%%%%%%%%%%%%%%%%%%%%%%%%%%%%%%%%%%%%%%%%%%%%%%%%%%%%%%%%%%%%%%%%%%%%%%%%%%%%%%%%%%%%%%%%%%%%

\subsection{$\Sigma(2,3,7,p,q)$ Preliminaries}
We proceed exactly as in the $\Sigma(2,3,5,p,q)$ case. Consider the Diophantine equation for $\Sigma(2,3,7,p,q)$ with solution $D = (e_0, 1, a, b, p', q')$, by reducing this modulo 42 we are able to solve for the Diophantine solutions based on $pq \mod 42.$

\begin{align*}
&pq \equiv 1 \mod 42 \implies D = (e_0,1,1,1,p',q') \quad \quad  &&pq \equiv 23 \mod 42 \implies D = (e_0,1,2,4,p',q') \\
&pq \equiv 5 \mod 42 \implies D = (e_0,1,2,3,p',q')\quad \quad &&pq \equiv 25 \mod 42 \implies D = (e_0,1,1,2,p',q') \\
&pq \equiv 11 \mod 42 \implies D = (e_0,1,2,2,p',q')\quad \quad &&pq \equiv 29 \mod 42 \implies D = (e_0,1,2,1,p',q') \\
&pq \equiv 13 \mod 42 \implies D = (e_0,1,1,6,p',q')\quad \quad &&pq \equiv 31 \mod 42 \implies D = (e_0,1,1,5,p',q') \\
&pq \equiv 17 \mod 42 \implies D = (e_0,1,2,5,p',q')\quad \quad &&pq \equiv 37 \mod 42 \implies D = (e_0,1,1,4,p',q') \\
&pq \equiv 19 \mod 42 \implies D = (e_0,1,1,3,p',q')\quad \quad &&pq \equiv 41 \mod 42 \implies D = (e_0,1,2,6,p',q') \\
\end{align*}

Note in all the above cases, $apq \equiv 1 \mod 3$ and $bpq \equiv 1 \mod 7$. 

We can also come up with a bound on $e_0$ using the Diophantine equation:
\begin{align*}
-1 &= 42pqe_0 + 21pq + 14apq + 6bpq + 42p'q + 42pq' \\
-42pqe_0 &= 21pq + 14apq+6bpq+42p'q + 42pq' + 1 \\
|e_0| &= \frac{21pq + 14apq + 6bpq + 42p'q + 42pq' + 1}{42pq}\\
|e_0| &= \frac{1}{2} + \frac{a}{3} + \frac{b}{7} + \frac{p'}{p} + \frac{q'}{q} + \frac{1}{42pq} \\
|e_0| &< \frac{1}{2} + \frac{2}{3} + \frac{6}{7} + \frac{p}{p} + \frac{q}{q} + \frac{1}{42*11*13} = \frac{4028}{1001} \approx 4.023\\
\end{align*}
Therefore $e_0 = -1, -2, -3, -4$. 

Using a similar simplification as we did for $\Sigma(2,3,5,p,q)$ we are able to compute $\Delta(xpq)$ and $\Delta(xpq \pm 1)$ for the $x$ values that satisfy the majority of our $\Sigma(2,3,7,p,q)$ cases. These values are presented in Table \ref{237pqfunctionvalues}. 
\begin{table}[h]
    \centering
    \begin{tabular}{|c||c|c|c|c|c|c|c|c|} 
        \hline 
         &$a = 1$ & $a = 1$ & $a = 1$ & $a = 1$ & $a = 2$ & $a = 2$ & $a = 2$ & $a = 2$ \\
         &$b = 1,2$ & $b = 3$ & $b = 4$ & $b = 5,6$ & $b = 1,2$ & $b = 3$ & $b = 4$ & $b = 5,6$ \\ \hline \hline 
         \rowcolor{lightgray!50}$\Delta(26pq - 1) =$  & $1 - |e_0|$ & $1 - |e_0|$ & $1 - |e_0|$ & $2 - |e_0|$&$2- |e_0|$ &$2- |e_0|$ &$2- |e_0|$ &$3- |e_0|$  \\ \hline
          \rowcolor{lightgray!50}$\Delta(26pq) =$  & 1 & 1 & 1 & 1 & 1 & 1 & 1 & 1  \\ \hline 
           \rowcolor{lightgray!50}$\Delta(26pq + 1) =$  & $|e_0| - 2$ & $|e_0| - 3$ & $|e_0| - 3$ & $|e_0| - 3$ & $|e_0| - 3$ &$ |e_0| - 4$ &$|e_0| - 4$ &$|e_0| - 4$  \\ \hline \hline 
           $\Delta(35pq - 1) =$  & $2 - |e_0|$ & $2 - |e_0|$ & $2 - |e_0|$ & $2 - |e_0|$&$3- |e_0|$ &$3- |e_0|$ &$3- |e_0|$ &$3- |e_0|$  \\ \hline
          $\Delta(35pq) =$  & 1 & 1 & 1 & 1 & 1 & 1 & 1 & 1  \\ \hline 
           $\Delta(35pq + 1) =$  & $|e_0| - 2$ & $|e_0| - 2$ & $|e_0| - 2$ & $|e_0| - 2$ & $|e_0| - 3$ &$ |e_0| - 3$ &$|e_0| - 3$ &$|e_0| - 3$  \\ \hline \hline 
           \rowcolor{lightgray!50}$\Delta(36pq - 1) =$  & $2 - |e_0|$ & $2 - |e_0|$ & $2 - |e_0|$ & $2 - |e_0|$&$2- |e_0|$ &$2- |e_0|$ &$2- |e_0|$ &$2- |e_0|$  \\ \hline
         \rowcolor{lightgray!50} $\Delta(36pq) =$  & 1 & 1 & 1 & 1 & 1 & 1 & 1 & 1  \\ \hline 
          \rowcolor{lightgray!50} $\Delta(36pq + 1) =$  & $|e_0| - 3$ & $|e_0| - 3$ & $|e_0| - 3$ & $|e_0| - 3$ & $|e_0| - 3$ &$ |e_0| - 3$ &$|e_0| - 3$ &$|e_0| - 3$  \\ \hline \hline 
          $\Delta(39pq - 1) =$  & $2 - |e_0|$ & $2 - |e_0|$ & $3 - |e_0|$ & $3 - |e_0|$&$2- |e_0|$ &$2- |e_0|$ &$3- |e_0|$ &$3- |e_0|$  \\ \hline
          $\Delta(39pq) =$  & 1 & 1 & 1 & 1 & 1 & 1 & 1 & 1  \\ \hline 
          $\Delta(39pq + 1) =$  & $|e_0| - 2$ & $|e_0| - 2$ & $|e_0| - 3$ & $|e_0| - 3$ & $|e_0| - 2$ &$ |e_0| - 2$ &$|e_0| - 3$ &$|e_0| - 3$  \\ \hline 
    \end{tabular}
    \caption{Key values of $\Delta(xpq)$ and $\Delta(xpq \pm 1)$ for $\Sigma(2,3,7,p,q)$}
    \label{237pqfunctionvalues}
\end{table}

%%%%%%%%%%%%%%%%%%%%%%%%%%%%%%%%%%%%%%%%%%%%%%%%%%%%%%%%%%%%%%%%%%%%%%%%%%%%%%%%%%%%%%%%%%%%%%%%%%%%%%%%%%%%%%%%%%%%%%%%%%%%%%%%%%%%%%%%%%%%%%%%%%%%%%%%%%%%%%%%%%%%%

\subsection{$\Sigma(2,3,7,p,q)$ Results}

Now we are ready to prove Lemma \ref{237pqlem}.

\lemBpq*

As with Lemma \ref{235pqlem}, we wish to find values that provide us with 2 consecutive $+1's$ in our $\Delta$-sequence that occur prior to $\nicefrac{N_0}{2}$. 
The values that satisfy this condition are found in Tables \ref{Compmod42} and \ref{1mod42}. \par

\begin{table}[h]
    \centering
    \begin{tabular}{|c||c|c|c|c|}
\hline
&$e_0 = -1$ & $e_1 = -2$ & $e_0 = -3$ & $e_0 = -4$\\ \hline \hline 
\multirow{2}{10em}{$pq \equiv 1 \mod 42$} & $26pq - 1$ & See Table \ref{1mod42} & $26pq$ & $26pq$\\
& $26pq$ &  & $26pq +1$ & $26pq +1$\\\hline
\multirow{2}{10em}{$pq \equiv 5 \mod 42$} & $26pq - 1$ & $26pq - 1$ & $39pq$ & $35pq$\\
& $26pq$ & $26pq$ & $39pq+1$ & $35pq+1$\\\hline 
\multirow{3}{10em}{$pq \equiv 11 \mod 42$} & $26pq - 1$ & $26pq - 1$ & $36pq$& $26pq$\\
& $26pq$ & $26pq$ & {\tiny$\vdots$} & $26pq+1$\\
&  &  & $36pq+3$& \\\hline 
\multirow{2}{10em}{$pq \equiv 13 \mod 42$} & $26pq - 1$ & $26pq - 1$ & $35pq$ & $26pq$ \\
& $26pq$ & $26pq$ & $35pq + 1$ & $26pq+1$ \\\hline 
\multirow{2}{10em}{$pq \equiv 17 \mod 42$} & $26pq - 1$ & $26pq - 1$  & $26pq - 1$ & $35pq$\\
& $26pq$ & $26pq$  & $26pq$ & $35pq+1$\\\hline 
\multirow{3}{10em}{$pq \equiv 19 \mod 42$} & $26pq - 1$ & $36pq$ & $35pq$ & $26pq$\\
& $26pq$ & $36pq + 1$ & $35pq+1$ & $26pq + 1$\\
&  & $36pq + 2$ &  & \\\hline 
\multirow{3}{10em}{$pq \equiv 23 \mod 42$} & $26pq - 1$ & $26pq - 1$ & $36pq$ & $35pq$\\
& $26pq$ & $26pq$ & {\tiny$\vdots$} & $35pq + 1$\\
&  &  & $36pq+3$ & \\\hline 
\multirow{3}{10em}{$pq \equiv 25 \mod 42$} & $26pq - 1$ & $36pq$ & $26pq$& $26pq$ \\
& $26pq$ & $36pq+1$ & $26pq+1$& $26pq+1$ \\
&  & $36pq+1$ & &  \\\hline 
\multirow{2}{10em}{$pq \equiv 29 \mod 42$} & $26pq - 1$ & $26pq - 1$  & $36pq$ & $26pq$\\
& $26pq$ & $26pq - 1$  & $36pq+1$ & $26pq+1$\\\hline 
\multirow{2}{7em}{$pq \equiv 31 \mod 42$} & $26pq - 1$ & $26pq - 1$ & $35pq$ &$26pq$\\ 
& $26pq$ & $26pq$ & $35pq+1$ &$26pq+1$\\ \hline 
\multirow{3}{10em}{$pq \equiv 37 \mod 42$} & $26pq - 1$ & $36pq$ & $35pq$ &$ 26pq$\\
& $26pq$ & {\tiny$\vdots$} & $35pq+1$ & $26pq+1$\\
&  & $36pq + 3$ &  &  \\ \hline
\multirow{2}{10em}{$pq \equiv 41 \mod 42$} & $26pq-1$ & $26pq-1$ & $26pq-1$ &$35pq$\\
& $26pq$ & $26pq$ & $26pq$ &$35pq+1$\\ \hline 
\end{tabular}
    \caption{General Cases for $Y = \Sigma(2,3,7,p,q)$}
    \label{Compmod42}
\end{table}

\begin{table}[h]
    \centering
    \begin{tabular}{|c||c|c|}
    \hline
         & $0<\nicefrac{p'}{p}<\nicefrac{1}{2}$ &$\nicefrac{1}{2}<\nicefrac{p'}{p}<1$  \\ \hline \hline
         \multirow{3}{7em}{$0<\nicefrac{q'}{q}<\nicefrac{1}{2}$} & $39pq$ & \cellcolor{green!25}$63pq - 42p'q - 1$\\
         & $39pq+1$ & \cellcolor{green!25}$63pq - 42p'q$\\
         & $39pq+2$ & \cellcolor{green!25}\\\hline
         \multirow{2}{7em}{$\nicefrac{1}{2}<\nicefrac{q'}{q}<1$} & \cellcolor{green!25}$63pq - 42pq' -1$ & \cellcolor{green!25}$63pq - 42pq' -1$ \\
         & \cellcolor{green!25}$63pq - 42pq'$ & \cellcolor{green!25}$63pq - 42pq'$ \\\hline
    \end{tabular}
    \caption{$pq \equiv 1 \mod 42$, $e_0 = -2$}
    \label{1mod42}
\end{table}

\begin{proof}
We must now show that the values  in Tables \ref{Compmod42} and \ref{1mod42} occur before $\nicefrac{N_0}{2}$. Since it is clear that $$ 26pq \leq 35pq \leq 36pq \leq 39pq < 39pq + 1 < 39pq + 2$$
we need only to check that $39pq + 1 \leq \nicefrac{N_0}{2}.$ It can be easily checked that this holds when $$p>6, \text{ and } q \geq \frac{42p + 4}{7p - 42},$$ which are satisfied in the following cases,
\begin{itemize}
    \item $p = 11,\quad q>25$,
    \item $p > 12$.
\end{itemize}
The finite cases that do not satisfy these inequalities may be easily checked to satisfy $$U \cdot HF_{red} \neq 0.$$ 

Thus it remains to address the following when $pq \equiv 1 \mod 42$ and $e_0 = -2$. First let $\nicefrac{1}{2}<\nicefrac{p'}{p}$, and $\nicefrac{q'}{q} < \nicefrac{1}{2}$. It can be easily checked that 
$$\Delta(63pq - 42pq' - 1) = 1 = \Delta(63pq - 42pq').$$
We must show that $63pq - 42pq' < \nicefrac{N_0}{2}$. This is true when $$p' > \frac{41pq + 42p + 42q}{84q}.$$
Using that $\nicefrac{1}{2} < \nicefrac{p'}{p}$, the above is satisfied when $p>42$ and $q > \frac{42p}{p-42}$. These inequalities are satisfied in the following cases,
\begin{multicols}{2}
\begin{itemize}
    \item $p > 84$
    \item $p = 43, \quad q>1806$
    \item $p = 47, \quad q>394$
    \item $p = 53, \quad q>202$
    \item $p = 55, \quad q>177$
    \item $p = 59, \quad q>145$
    \item $p = 61, \quad q>134$
    \item $p = 65, \quad q>118$
    \item $p = 67, \quad q>112$
    \item $p = 71, \quad q>102$
    \item $p = 73, \quad q>98$
    \item $p = 79, \quad q>89$
\end{itemize}
\end{multicols}
The case where $\nicefrac{p'}{p} < \nicefrac{1}{2}$ and $\nicefrac{q'}{q} > \nicefrac{1}{2}$ can be checked similarly. Thus our proof is complete. \\
\end{proof}

%%%%%%%%%%%%%%%%%%%%%%%%%%%%%%%%%%%%%%%%%%%%%%%%%%%%%%%%%%%%%%%%%%%%%%%%%%%%%%%%%%%%%%%%%%%%%%%%%%%%%%%%%%%%%%%%%%%%%%%%%%%%%%%%%%%%%%%%%%%%%%%%%%%%%%%%%%%%%%%%%%%%%

\subsection{$\Sigma(2,3,11,13,q)$ Preliminaries}
Before we proceed, we must first complete some algebraic preliminaries. This will follow the general idea of the previous two sections, but we take a slightly different approach. Consider $Y = \Sigma(2,3,11,13,p),$ with Diophantine solutions $D = (e_0, 1, a, b, c, p')$. Then the Diophantine equation simplifies as follows 
$$-1 = 858pe_0 + 429p + 286ap + 78bp + 66cp + 858p'. $$
Reducing this modulo 3, 11 and 13 we get that
$$qp \equiv 2 \mod 3, \quad bp \equiv 10 \mod 11, \quad cp \equiv 12 \mod 13.$$
We can also come up with a bound on $e_0$ using the Diophantine equation:
\begin{align*}
-1 &= 858pe_0 + 429p + 286ap + 78bp + 66cp + 858p'\\
|e_0| &= \frac{1}{2} + \frac{a}{3} + \frac{b}{11} + \frac{c}{13} + \frac{p'}{p} + \frac{1}{858p}\\
|e_0| &< \frac{1}{2} + \frac{2}{3} + \frac{10}{11} + \frac{12}{13} + \frac{p}{p} \approx 3.99
\end{align*}
Therefore $e_0 = -1, -2, -3$. \par 

%%%%%%%%%%%%%%%%%%%%%%%%%%%%%%%%%%%%%%%%%%%%%%%%%%%%%%%%%%%%%%%%%%%%%%%%%%%%%%%%%%%%%%%%%%%%%%%%%%%%%%%%%%%%%%%%%%%%%%%%%%%%%%%%%%%%%%%%%%%%%%%%%%%%%%%%%%%%%%%%%%%%%

\subsection{$\Sigma(2,3,11,13,q)$ Results}
We are now ready to prove the following.

\lemCpq

\begin{proof}
Consider the following,
\begin{align*}
    \Delta(781p) &= 1+|e_0|(781p) - \left \lceil \frac{781p}{2}\right \rceil - \left \lceil \frac{781ap}{3}\right \rceil - \left \lceil \frac{781bp}{11}\right \rceil - \left \lceil \frac{781cp}{13}\right \rceil - \left \lceil \frac{781pp'}{p}\right \rceil 
\end{align*}
Using that $$qp \equiv 2 \mod 3, \quad bp \equiv 10 \mod 11, \quad cp \equiv 12 \mod 13,$$ and that $$781 \equiv 1 \mod 3, \quad 781 \equiv 0 \mod 11, \quad 781 \equiv 1 \mod 13,$$ we know the value of the ceiling functions. Now we solve the Diophantine equation for $p',$
$$p' = \frac{-1}{858} - pe_0 - \frac{p}{2} - \frac{ap}{3} - \frac{bp}{11} - \frac{cp}{13}.$$
Now we are able to simplify $\Delta(781p)$ fully.
\begin{align*}
    \Delta(781p) &= 1+|e_0|(781p) -  \frac{781p + 1}{2} -  \frac{781ap+1}{3} - 71bp -  \frac{781cp+1}{13} - 781p'\\
    &= 1+|e_0|(781p) -  \frac{781p + 1}{2} -  \frac{781ap+1}{3} - 71bp -  \frac{781cp+1}{13}\\
    &\quad\quad- 781\left(\frac{-1}{858} - pe_0 - \frac{p}{2} - \frac{ap}{3} - \frac{bp}{11} - \frac{cp}{13}\right)\\
    &= 1+|e_0|(781p) -  \frac{781p + 1}{2} -  \frac{781ap+1}{3} - 71bp-  \frac{781cp+1}{13}\\
    &\quad\quad+ \frac{781}{858} +781pe_0 +\frac{781p}{2} + \frac{781qp}{3} + \frac{781bp}{11} + \frac{781cp}{13}\\
    &= 1 + \frac{781}{858} - \frac{1}{2} - \frac{1}{3} - \frac{1}{13}\\
    &=1
\end{align*}

Using this same process we have that

\[\Delta(781p + 1) = 
    \begin{dcases}
        |e_0| - 1, & a=1, c=1 \\
        |e_0| - 2, & a=1 \text{ and } c \neq 1, \text{ or }  a = 2, c=1 \\
       |e_0| - 3, & a = 2, c \neq 1 \\
    \end{dcases}
\]

\[\Delta(781p - 1) = 
    \begin{dcases}
        2 - |e_0|, & a=1, c \neq 12 \\
        3 - |e_0|, & a=1 \text{ and } c = 12, \text{ or }  a = 2, c \neq 12 \\
       4 -|e_0|, & a = 2, c = 12 \\
    \end{dcases}
\]

We can see in Table \ref{13pcase} the values that give us two consecutive ones, given conditions on $|e_0|$, $a$ and $c$. Since $781p - 1 < 781p < 781p+1,$ we need to show that $781p+1<\nicefrac{N_0}{2}.$ This is easily checked to be true when $p\geq 6$, which is always true for $Y= \Sigma(2,3,11,13,p).$ It remains to check the following two cases,

\begin{table}[h]
    \centering
    \begin{tabular}{|c||c|c|c|}
\hline
&$e_0 = -1$ & $e_1 = -2$ & $e_0 = -3$  \\ \hline \hline 
\multirow{2}{10em}{$a = 1, c = 1$} & $781p - 1$ & $781p$ & $781p$\\
& $781p $ &  $781p + 1$ & $781p + 1$\\\hline
\multirow{2}{10em}{$a = 1, c = 12$} & $781p - 1$ & $781p - 1$ & $781p$\\
& $781p$ &$781p$ & $781p + 1$ \\\hline 
\multirow{2}{10em}{$a = 1, c \neq 1, 12$}& $781p - 1$& & $781p$\\
& $781p $&  & $781p + 1$\\ \hline
\multirow{2}{10em}{$a = 2, c = 1$} & $781p - 1$ & $781p - 1$ & $781p$  \\
&$781p$ &$781p$ & $781p + 1$ \\\hline 
\multirow{2}{10em}{$a = 2, c = 12$} & $781p - 1$ & $781p - 1$  & $781p - 1$\\
& $781p $ & $781p $ & $781p$ \\\hline 
\multirow{2}{10em}{$a = 2, c \neq 1, 12$} & $781p - 1$ & $781p - 1$&  \\
& $781p$ & $781p$ &\\\hline
\end{tabular}
    \caption{General Cases for $Y = \Sigma(2,3,11,13,p)$}
    \label{13pcase}
\end{table}

\begin{enumerate}
\item $|e_0| = 2$, $a =1$, $c \neq 1, 12$
\item $|e_0| = 3$, $a = 1$, $c \neq 1,12$
\end{enumerate}

\paragraph{Case 1:} It can be checked that $\Delta(1287p - 858p') = 1$ and $\Delta(2387p - 858p') = 3-|e_0| = 1$. Now it remains to check that $1287p - 858p' < \nicefrac{N_0}/2$. We use the Diophantine equation to show that 
$$p' \geq \frac{505p-1}{858},$$
and then plug this into our above inequality. We see that $1287p - 858p' < \nicefrac{N_0}/2$ is satisfied when $p > \frac{860}{151} \approx 5.6,$ which is trivially true, and so this case is complete. 

\paragraph{Case 2:} It is easily checked that $\Delta(1573p - 858p' - 1) = 1$ and $\Delta(1573p - 858p') = 1$. We find that $1573p-858p' < \nicefrac{N_0}{2}$ when $p \geq 6$. 

\end{proof}

\section{Mapping Cone}\label{MappingCone}
In this section we give the proper notation and theorems used to set up the mapping cone for rational surgery on a knot in $S^3$. The mapping cone formula was introduced by Ozsv\'{a}th and Szab\'{o} in \cite{OSMC}, but we follow Gainullin's general structure in \cite{GainMC}. \par

\subsection{Background}
Given a knot $K$ in $S^3$ we associate a doubly-filtered complex $C = CFK^{\infty}(K).$ The generators of this complex are denoted $[\mathbf{x},i,j]$, where $(i,j)\in\Z \times \Z$ is the filtration. We have that $C$ is homotopy equivalent as a filtered complex to a complex where all filtration preserving differentials are trivial \cite{Rasmu}. Therefore we can replace the group, viewed as a chain complex, with its homology at each filtration level. We will work with this object, called the reduced complex. \par

This complex $C$ has an associated $U$-action, which corresponds to translation by the vector $(-1, -1).$ $C$ is invariant under this shift, so the group at filtration level $(i,j)$ is the same as the one at $(i-1, j-1)$, and we can view $U$ as the identity map between these. We know that $U$ is a chain map, and $U$ is invertible, so $C$ is an $\mathbb{F}[U,U^{-1}]$ - module. It follows that $C$ is generated by elements at filtration level $i=0$, and we will refer to this complex at filtration level $(0,j)$ as $\widehat{HFK}(K,j)$. \par

We can view $C$ as the complex used to compute $HF^{\infty}(S^3),$ where the knot provides an additional filtration. Then we can use the grading on $HF^{\infty}(S^3)$ to obtain a grading on $C$. The associated $U$-action decreases this grading by 2. \par

We can now define the following quotient complexes of $C$. 
$$A_k^+(K) = C\{i \geq 0 \text{ or } j \geq k\}, \quad k \in \mathbb{Z},$$
and 
$$B^+ = C\{i \geq 0\} \cong CF^+(S^3).$$
We also have two chain maps $v_k\colon A_k^+(K) \to B^+$, and $h_k\colon A_k^+(K) \to B^+$. The map $v_k$ is simply projection, sending all generators with $i>0$ to $0$, and the identity elsewhere. The map $h_k$ first projects onto $C\{j \geq k\},$ then multiplies by $U^k$, and then identifies $C\{j\geq 0\}$ with $C\{i \geq 0\}$ via a chain homotopy equivalence. This homotopy equivalence exists because both complexes represent $CF^+(S^3).$ It is known that $v_k$ is an isomorphism if $k \geq g(K)$ and $h_k$ is an isomorphism if $k \geq -g(K)$, where $g(K)$ is the genus of the knot \cite{OSMC}.\par

Now we can define chain complexes 
$$\mathcal{A}^+_{i, p/q}(K) = \bigoplus_{n \in \mathbb{Z}}(n, A^+_{\left \lfloor \frac{i+pn}{q}\right \rfloor}(K)),$$ and 
$$\mathcal{B}^+ = \bigoplus_{n \in \mathbb{Z}}(n,B^+).$$
We use the index $n$ just to distinguish between different copies of the same group. Then we have a chain map $$D^+_{i,p/q}\colon \mathcal{A}_{i,p/q}^+ \to \mathcal{B}^+,$$
where $D^+_{i,p/q}$ is given by the appropriate sums of $v_k$ and $h_k$. We require $v_k \colon (n, A^+_k(K) \to (n, B^+)$, and $h_k \colon (n, A^+_k(K) \to (n+1, B^+)$. We can define this map explicitly as follows,
$$D^+_{i,p/q}(\{k,a_k\})_{k \in \mathbb{Z}}) = \{(k,b_k)\}_{k \in\mathbb{Z}}, \text{ where } \quad b_k = v^+_{\left \lfloor  \frac{i+pk}{q}\right \rfloor}(a_k) + h^+_{\left \lfloor \frac{i+p(k-1)}{q}\right \rfloor}(a_{k-1}).$$
The complexes $A_k^+(K)$ and $B^+$ inherit a relative $\mathbb{Z}$-grading from $C$. Now let $\mathbb{X}^+_{i,p/q}$ denote the mapping cone of $D^+_{i,p/q}$. This has a relative $\mathbb{Z}$-grading by requiring $D^+_{i,p/q}$ decreases the grading by 1. 

\begin{theorem}[{O}zsv\'ath-{S}zab\'o, \cite{OSMC}]
There is a relatively graded isomorphism of $\F[U]$-modules $$H_*\left(\X^+_{i,p/q}\right) \cong HF^+\left(S^3_{p/q}(K), i\right).$$
\end{theorem}

Note that $i$ indexes the Spin$^c$-structure of $HF^+\left(S^3_{p/q}(K),i\right)$. Since we are working with Seifert fibered integral homology spheres, our spaces will only have one Spin$^c$-structure, thus we will often drop the $i$, as this index is not important for our purposes. \par

The picture you should have in your head for the mapping cone is depicted in Figure \ref{GenMCpic} for $\nicefrac{2}{3}$ surgery of some knot in $S^3$. We have $3$ copies of each group, which the subscripts denote, and the $h_k$ maps go over $2$. For general $\nicefrac{p}{q}$ surgery we will have $q$ copies of each group and the $h_k$ maps will go over $p$. \par

Since we know that $v_k$ is an isomorphism if $k \geq g(K)$ and $h_k$ is an isomorphism if $k \geq -g(K)$, when we are working with a specific knot, we know much of the mapping cone is acyclic. We can delete this acyclic part, which we will refer to as the truncated mapping cone. This truncated mapping cone is depicted for $\nicefrac{2}{3}$ surgery of some genus 2 knot in $S^3$ in Figure \ref{TrunMCpic}

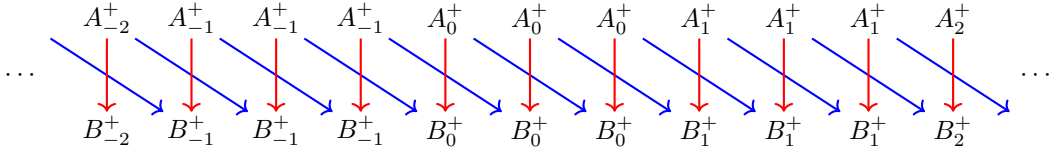
\begin{figure}
    \centering
    \begin{tikzpicture}[thick,scale=0.75, every node/.style={scale=1}]
    \node at (-7.5,1) {$A_{-2}^+$};
    \node at (-6,1) {$A_{-1}^+$};
    \node at (-4.5,1) {$A_{-1}^+$};
    \node at (-3,1) {$A_{-1}^+$};
    \node at (-1.5,1) {$A_{0}^+$};
    \node at (0,1) {$A_{0}^+$};
    \node at (1.5,1) {$A_{0}^+$};
    \node at (3,1) {$A_{1}^+$};
    \node at (4.5,1) {$A_{1}^+$};
    \node at (6,1) {$A_{1}^+$};
    \node at (7.5,1) {$A_{2}^+$};
    \node at (-7.5,-1) {$B_{-2}^+$};
    \node at (-6,-1) {$B_{-1}^+$};
    \node at (-4.5,-1) {$B_{-1}^+$};
    \node at (-3,-1) {$B_{-1}^+$};
    \node at (-1.5,-1) {$B_{0}^+$};
    \node at (0,-1) {$B_{0}^+$};
    \node at (1.5,-1) {$B_{0}^+$};
    \node at (3,-1) {$B_{1}^+$};
    \node at (4.5,-1) {$B_{1}^+$};
    \node at (6,-1) {$B_{1}^+$};
    \node at (7.5,-1) {$B_{2}^+$};
    \node at (-9,0) {$\cdots$};
    \node at (9,0) {$\cdots$};
    
    \draw [blue, ->](-7,.65) -- (-5,-.65);
    \draw [blue, ->](-8.5,.65) -- (-6.5,-.65);
    \draw [blue, ->](-4,.65) -- (-2,-.65);
    \draw [blue, ->](-1,.65) -- (1,-.65);
    \draw [blue, ->](2,.65) -- (4,-.65);
    \draw [blue, ->](5,.65) -- (7,-.65);
    \draw [blue, ->](-5.5,.65) -- (-3.5,-.65);
    \draw [blue, ->](-2.5,.65) -- (-.5,-.65);
    \draw [blue, ->](.5,.65) -- (2.5,-.65);
    \draw [blue, ->](3.5,.65) -- (5.5,-.65);
    \draw [blue, ->](6.5,.65) -- (8.5,-.65);
    
    \draw[red,->] ((-7.5,.65) -- (-7.5,-.65);
    \draw[red,->] ((-6,.65) -- (-6,-.65);
    \draw[red,->] ((-4.5,.65) -- (-4.5,-.65);
    \draw[red,->] ((-3,.65) -- (-3,-.65);
    \draw[red,->] ((-1.5,.65) -- (-1.5,-.65);
    \draw[red,->] ((0,.65) -- (0,-.65);
    \draw[red,->] ((7.5,.65) -- (7.5,-.65);
    \draw[red,->] ((6,.65) -- (6,-.65);
    \draw[red,->] ((4.5,.65) -- (4.5,-.65);
    \draw[red,->] ((3,.65) -- (3,-.65);
    \draw[red,->] ((1.5,.65) -- (1.5,-.65);
    \end{tikzpicture}
    \caption{The Mapping cone for $\nicefrac{2}{3}$-surgery on a knot $K \subset S^3$. The $v_k$ maps are red, $h_k$ maps are blue.}
    \label{GenMCpic}
\end{figure}

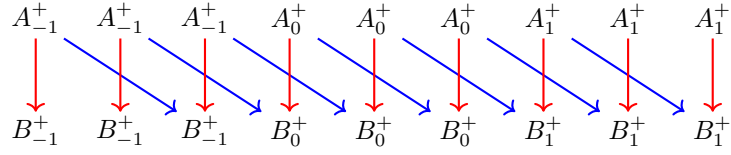
\begin{figure}
    \centering
    \begin{tikzpicture}[thick,scale=0.75, every node/.style={scale=1}]
    \node at (-6,1) {$A_{-1}^+$};
    \node at (-4.5,1) {$A_{-1}^+$};
    \node at (-3,1) {$A_{-1}^+$};
    \node at (-1.5,1) {$A_{0}^+$};
    \node at (0,1) {$A_{0}^+$};
    \node at (1.5,1) {$A_{0}^+$};
    \node at (3,1) {$A_{1}^+$};
    \node at (4.5,1) {$A_{1}^+$};
    \node at (6,1) {$A_{1}^+$};
    \node at (-6,-1) {$B_{-1}^+$};
    \node at (-4.5,-1) {$B_{-1}^+$};
    \node at (-3,-1) {$B_{-1}^+$};
    \node at (-1.5,-1) {$B_{0}^+$};
    \node at (0,-1) {$B_{0}^+$};
    \node at (1.5,-1) {$B_{0}^+$};
    \node at (3,-1) {$B_{1}^+$};
    \node at (4.5,-1) {$B_{1}^+$};
    \node at (6,-1) {$B_{1}^+$};

    \draw [blue, ->](-4,.65) -- (-2,-.65);
    \draw [blue, ->](-1,.65) -- (1,-.65);
    \draw [blue, ->](2,.65) -- (4,-.65);
    \draw [blue, ->](-5.5,.65) -- (-3.5,-.65);
    \draw [blue, ->](-2.5,.65) -- (-.5,-.65);
    \draw [blue, ->](.5,.65) -- (2.5,-.65);
    \draw [blue, ->](3.5,.65) -- (5.5,-.65);
    
    \draw[red,->] ((-6,.65) -- (-6,-.65);
    \draw[red,->] ((-4.5,.65) -- (-4.5,-.65);
    \draw[red,->] ((-3,.65) -- (-3,-.65);
    \draw[red,->] ((-1.5,.65) -- (-1.5,-.65);
    \draw[red,->] ((0,.65) -- (0,-.65);
    \draw[red,->] ((6,.65) -- (6,-.65);
    \draw[red,->] ((4.5,.65) -- (4.5,-.65);
    \draw[red,->] ((3,.65) -- (3,-.65);
    \draw[red,->] ((1.5,.65) -- (1.5,-.65);
    \end{tikzpicture}
    \caption{The truncated mapping cone for $\nicefrac{2}{3}$-surgery on a genus 2 knot $K \subset S^3$. The $v_k$ maps are red, $h_k$ maps are blue.}
    \label{TrunMCpic}
\end{figure}

Now let $\mathbf{A}^+_{i,p/q}(K) = H_*(A^+_k(K))$ and $ \mathbf{B}^+ = H_*(B^+),$ $\mathbb{A}^+_{i,p/q}(K) = H_*(\mathcal{A}^+_{i,p/q}(K))$, $\mathbb{B}^+ = H_*(\mathcal(B)^+)$ and let $\mathbf{v}_k, \mathbf{h}_k$ and $\mathbf{D}^+_{i,p/q}$ denote the maps induced by $v_k$, $h_k$ and $D^+_{i,p/q}$ on homology respectively.  \par 
We have that the short exact sequence
$$0 \longrightarrow \mathcal{B}^+ \overset{\iota}\longrightarrow\X^+_{i,p/q} \overset{j}\longrightarrow \mathcal{A}^+_{i,p/q}(K) \longrightarrow 0$$
induces the exact triangle
\begin{center}
\begin{tikzcd}
\mathbb{A}^+_{i,p/q(K)} \arrow[r, "\mathbf{D}^+_{i,p/q}"]
& \mathbb{B}^+ \arrow[d, "\iota_*"] \\
& H_*\left(\X^+_{i,p/q}\right) \cong HF^+\left(S^3_{p/q}(K),i\right) \arrow[lu, "j_*"].
\end{tikzcd}
\end{center}

All maps above are $U$-equivariant. From the exact triangle we have the following,

\begin{cor}(Gainullin, \cite[Section 2]{GainMC})
If the surgery slope $\nicefrac{p}{q}$ is positive, then the map $\mathbf{D}^+_{i,p/q}$ will be surjective, so $HF^+\left(S^3_{p/q}(K), i\right) \cong \ker\left(\mathbf{D}^+_{i,p/q}\right).$
\end{cor}

We need to establish some important decompositions of the maps given earlier. First we have that we can decompose $\mathbf{A}_k^+(K)$ as $\mathbf{A}^+_k(K) \cong \mathbf{A}^T_k(K) \oplus \mathbf{A}^{red}_k(K)$, where $\mathbf{A}^{red}_k(K)$ is a finite-dimensional vector space in the kernel of some power of $U$, and $\mathbf{A}^T_k(K) \cong \mathcal{T}^+$. We also have an isomorphism $\phi \colon \mathbf{A}_0^+(K) \to \mathbf{A}_0^+(K)$ that is $U$-equivariant, grading preserving, and satisfies $\mathbf{v}_k \circ \phi = \mathbf{h}_k,$ where we view $\mathbf{v}_k$ and $\mathbf{h}_k$ as maps into $CF^+(S^3),$ and $\phi$ as an isomorphism from $CF^+(S^3)$ to itself. Note, we will often drop $K$ from the above notation and refer to $\mathbf{A}_k^+(K)$ as simply $\mathbf{A}_k^+$.\par
 We have that 
 $$\mathbf{D}^+_{i,p/q} = \mathbf{D}^T_{i,p/q} \oplus \mathbf{D}^{red}_{i,p/q}$$
  where the first map is the restriction of $\mathbf{D}^+_{i,p/q}$ to $\mathbb{A}^T_{i,p/q}(K) = \bigoplus_{n \in \Z}\mathcal{T}^+$ and the second one is the restriction to $\mathbb{A}^{red}_{i,p/q}(K)$. Similarly we have restrictions of $\mathbf{v}_k$ and $\mathbf{h}_k$ to $\mathcal{T}^+$ which will just be some powers of $U$. These powers of $U$ will be denoted by $V_k$ and $H_k$. We will need the following properties of these integers, (see \cite[Section 2]{NiWuCosSur}, \cite[Section 7]{Rasmu}, \cite[Lemma 2.5]{HLZ})
 
 \begin{align}
&V_k \geq V_{k+1} \text{ and } H_k \leq H_{k+1}, \quad \forall k \in \Z,\label{EQ3}\\
&V_k = H_{-k},\quad  \forall k \in \Z, \\
&V_k \to + \infty \text{ as } k \to - \infty \quad \text{ and } \quad  H_k \to + \infty \text{ as } k \to +\infty \\
&V_k = 0 \text{ for } k \geq g(K)\quad  \text{ and } \quad H_k = 0 \text{ for } k \leq -g(K),\label{EQ2}\\
&V_k - 1 \leq V_{k+1} \leq V_k, \forall k \in \Z,\\
&H_k = V_k + k, \forall k \in \mathbb{Z}. \label{EQ1}
\end{align}

Finally, we state a Corollary of Gainullin that we will reference frequently in the next section. 

\begin{cor}(\cite[Corollary 30]{GainMC})
$U^g \cdot \mathbb{A}^{red}_{i,p/q} = 0.$
\label{genusthm}
\end{cor}

\subsection{Results}
The goal of this section is to prove Theorem 3, but first we must prove some preliminary results about the $\mathbf{v_k}$ and $\mathbf{h_k}$ maps, and subsequently, $V_k$ and $H_k$. 

\begin{lemma}
Let $K$ be a knot in $S^3$ with genus $g$. Then $H_i \leq g$ when $i \leq g$ and $V_i \leq g$ when $i \geq 0$. 
\label{VH}
\end{lemma}

\begin{proof}
By \eqref{EQ1} $H_g = V_g+g$ and by \eqref{EQ2} $V_g = 0$. Therefore $H_g = g$. By \eqref{EQ3} we have that $H_i \leq g$ for $i \leq g$. Since $V_0 = H_0$ by \eqref{EQ1}, and $H_0 \leq g$, we know that $V_0 \leq g$. Then by \eqref{EQ3} we have that $V_i \leq g$ when $i \geq 0$. 
\end{proof}

\begin{lemma}
For a genus 1 knot $K$ we have that $U\cdot HF_{red}(S^3_{\pm 1}(K)) = 0$.
\label{onesurg}
\end{lemma}

\begin{proof}
By the mapping cone formula $HF^+(S^3_1(K)) = H_*(A_0^+)$ and by Corollary \ref{genusthm} we have $$U^{g(K)} \cdot H_*(A_0^+) = 0,$$ where $g(K)$ is the genus of $K$, thus $$U \cdot HF_{red}(S^3_1(K)) = 0.$$ We know that $S^3_{1/q}(K) = - S^3_{-1/q}(\overline{K})$, where $\overline{K}$ denotes the the reflection of $K$, \cite{OSSurg}. Thus it follows that $$S^3_{+1}(K) = -S^3_{-1}(\overline{K}).$$ Since $K$ and $\overline{K}$ are both genus 1 knots it follows that $$U \cdot HF_{red}(S^3_{-1}(K)) = 0,$$ by the arguments for +1 surgery.
\end{proof}

Before we proceed, we must introduce the concept of levels, a way to refer to the relative grading of $\mathbf{A}_k^+$ and $\mathbf{B}_k^+$. We define level 0 to contain the bottom of $\mathcal{T}^+$, and those elements that have this same grading, call this grading $g$. Then level 1 will consist of those elements in grading $g+2$, level 2 will consist of those elements in grading $g+4$. Thus in general elements in grading $g + 2k$, will be in level $k$. Elements in different parity of grading will be in half levels. So an element in grading $g+2k+1$ will be in level $\frac{2k+1}{2}$.

\begin{example}
Consider $\mathbf{A}_0^+$ in figure \ref{levelspic}. The $a_i$'s are the vertices in $\mathbf{A}_0^+$, while the other elements are in $\mathbf{A}_0^{red}$. We see that $a_0$ is the bottom of the tower, and thus is in level 0. The only elements in different parity of grading are $w$ and $z$.  We can then say the other elements are in the following levels,
\begin{multicols}{3}
\begin{itemize}
\centering
\item[Level 0:] $a_0, v, y$
\item[Level 1:] $a_1, x$
\item[Level 2:] $a_2$
\item[Level 3:] $a_3,u$
\item[Level $\frac{3}{2}$:] $w$
\item[Level $\frac{5}{2}$:] $z$
\end{itemize}
\end{multicols}

\begin{figure}[h]
\centering
\begin{tikzpicture}[thick,scale=0.75, every node/.style={scale=0.75}]
\node[left] at (-5,-1) {k+6};
\node[left] at (-5,-2) {k+4};
\node[left] at (-5,-3) {k+2};
\node[left] at (-5,-4) {k};
 \draw[->,red]
    (-.3,-.9)[anchor = east] [out=200, in=150] to  (-.3,-1.9)[anchor = east];
\node[left] at (-.5,-1.5) {$U$};
\draw[gray, dashed] (-5,-1)--(5,-1);
\draw[gray, dashed] (-5,-2)--(5,-2);
\draw[gray, dashed] (-5,-3)--(5,-3);
\draw[gray, dashed] (-5,-4)--(5,-4);

\filldraw[black] (0,-4) circle (2pt) node[anchor=south] {$a_0$};
\filldraw[black] (0,-3) circle (2pt) node[anchor=south] {$a_1$};
\filldraw[black] (0,-2) circle (2pt) node[anchor=south] {$a_2$};
\filldraw[black] (0,-1) circle (2pt) node[anchor=south] {$a_3$};

\node[hackennode] (a0) at (0,-.2) {};
\node[hackennode] (a0) at (0,.2) {};
\node[hackennode] (a0) at (0,0) {};

\filldraw[black] (-3,-1) circle (2pt) node[anchor=south] {$u$};
\filldraw[black] (-2,-4) circle (2pt) node[anchor=south] {$v$};
\filldraw[black] (-1,-2.5) circle (2pt) node[anchor=south] {$w$};
\filldraw[black] (1,-3) circle (2pt) node[anchor=south] {$x$};
\filldraw[black] (2,-4) circle (2pt) node[anchor=south] {$y$};
\filldraw[black] (3,-1.5) circle (2pt) node[anchor=south] {$z$};

\node (label) at (0,-4.75) {$\mathbf{A}_0^+$};

\end{tikzpicture}
\caption{Levels of $\mathbf{A}_0^+$}
\label{levelspic}
\end{figure}
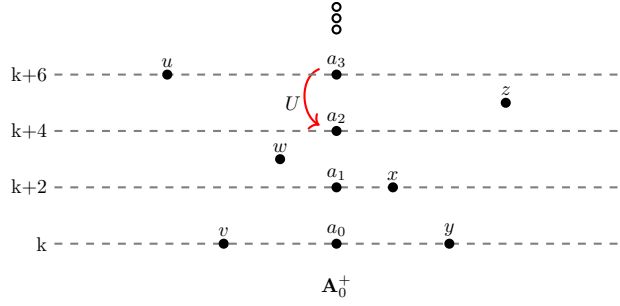

\end{example}

\begin{remark}
Note that linear combinations of elements in half levels will always be sent to 0 under the $U$-action by Corollary \ref{genusthm}. Thus elements in odd levels can never contribute to $U$ of $HF_{red}$ being nonzero, so we need not discuss them to prove Theorem \ref{gen1thm}. 
\end{remark}

Now we use Lemma \ref{onesurg} to restrict the allowable maps for $\mathbf{v}_0$ and $\mathbf{h}_0$. Consider figure \ref{-1surg}, which depicts the truncated mapping cone for -1 surgery on a genus 1 knot. This figure shows the labeling used to refer to elements of $\mathbf{A}_0^T$, $\mathbf{B}_0^T$ and $\mathbf{B}_{-1}^T$, where the subscript of the element refers to the element it is in. 

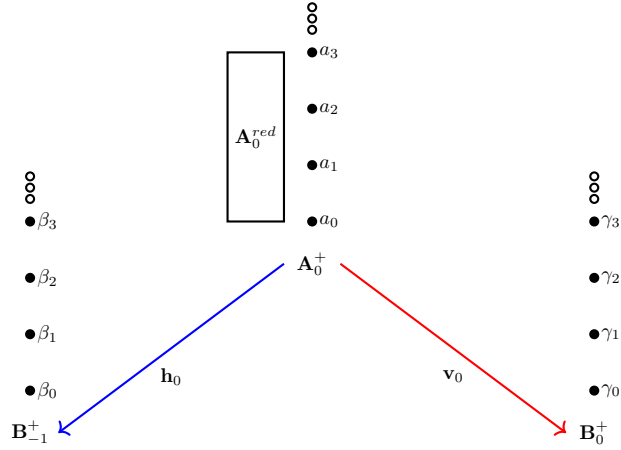
\begin{figure}[h]
\centering
\begin{tikzpicture}[thick,scale=0.75, every node/.style={scale=0.75}]
\node[hackennode] (a0) at (0,4.4) {};
\node[hackennode] (a0) at (0,4.6) {};
\node[hackennode] (a0) at (0,4.8) {};
\node (label) at (0,.25) {$\mathbf{A}_0^+$};
\filldraw[black] (0,1) circle (2pt) node[anchor=west] {$a_0$};
\filldraw[black] (0,2) circle (2pt) node[anchor=west] {$a_1$};
\filldraw[black] (0,3) circle (2pt) node[anchor=west] {$a_2$};
\filldraw[black] (0,4) circle (2pt) node[anchor=west] {$a_3$};
\draw (-.5,1) rectangle (-1.5,4);
\node (label) at (-1,2.5) {$\mathbf{A}_0^{red}$};

\node[hackennode] (a0) at (-5,1.8) {};
\node[hackennode] (a0) at (-5,1.6) {};
\node[hackennode] (a0) at (-5,1.4) {};
\node (label) at (-5,-2.75) {$\mathbf{B}_{-1}^+$};
\filldraw[black] (-5,-2) circle (2pt) node[anchor=west] {$\beta_0$};
\filldraw[black] (-5,-1) circle (2pt) node[anchor=west] {$\beta_1$};
\filldraw[black] (-5,0) circle (2pt) node[anchor=west] {$\beta_2$};
\filldraw[black] (-5,1) circle (2pt) node[anchor=west] {$\beta_3$};
%\draw (-5.5,-2) rectangle (-6.5,1);
%\node (label) at (-6,-.5) {$\mathbf{B}_0^{red}$};

\node[hackennode] (a0) at (5,1.8) {};
\node[hackennode] (a0) at (5,1.6) {};
\node[hackennode] (a0) at (5,1.4) {};
\node (label) at (5,-2.75) {$\mathbf{B}_0^+$};
\filldraw[black] (5,-2) circle (2pt) node[anchor=west] {$\gamma_0$};
\filldraw[black] (5,-1) circle (2pt) node[anchor=west] {$\gamma_1$};
\filldraw[black] (5,0) circle (2pt) node[anchor=west] {$\gamma_2$};
\filldraw[black] (5,1) circle (2pt) node[anchor=west] {$\gamma_3$};
%\draw (5.6,-2) rectangle (6.6,1);
%\node (label) at (6.1,-.5) {$\mathbf{B}_0^{red}$};

\draw [blue, ->] (-.5,.25) -- (-4.5,-2.75); 
\node (label) at (-2.5,-1.75) {$\mathbf{h}_0$};
\draw [red, ->] (.5,.25) -- (4.5,-2.75); 
\node (label) at (2.5,-1.75) {$\mathbf{v}_0$};

\end{tikzpicture}
\caption{-1 Surgery Mapping Cone}
\label{-1surg}
\end{figure}

\begin{remark}
Note we have only drawn a box to represent $\mathbf{A}_0^{red}$, and neither of the $B^+$ complexes have a reduced part as they represent $CF^+(S^3)$.
\end{remark}

\begin{lemma}
Let $x_i$ be in level $i$ of $\mathbf{A}_0^{red}$. Then $\mathbf{v_0}(x_i) = 0 = \mathbf{h_0}(x_i)$ for $i \geq V_0 + 1$, or $i < 0$.
\label{lvls1}
\end{lemma}

\begin{proof}
We know that $\mathbf{v_0}$ must send $x_i$ to an element in level $i-(V_0)$ of $B_0^+$, due to the grading restriction on $\mathbf{v}_0$. Therefore if $i<0$, $\mathbf{v}_0(x_i) = 0$ as there are no elements in $B_0^+$ at level $i$ for negative $i$. Now consider $i \geq V_0 + 1$. Assume indirectly that $\mathbf{v_0}(x_i) = \gamma_{i-(V_0)}$. Then we have that $$U(\mathbf{v_0}(x_i)) = U\cdot \gamma_{(i-V_0)} = \gamma_{(i - V_0- 1)}.$$ 
But by the $U$-equivariance of $\mathbf{v}_0$ we also have 
$$U(\mathbf{v}_0(x_i)) = \mathbf{v}_0(U\cdot x_i) = \mathbf{v}_0(0) =  0.$$
Thus we have a contradiction so $\mathbf{v}_0(x_i) = 0.$ It can be shown that $\mathbf{h}_0(x_i) = 0$ using the same arguments. 
\end{proof}

\begin{lemma}
Let $V_0 = H_0 = 1$ and $x$ be in level $1$ of $\mathbf{A}_0^{red}$.  Then $\mathbf{v_0}(x) = 0 = \mathbf{h_0}(x)$. 
\label{lvls2}
\end{lemma}

\begin{proof} Due to the grading restrictions on $\mathbf{v}_0$ and $\mathbf{h}_0$ we know that $\mathbf{v}_0(x) = 0$ or $\mathbf{v}_0(x) = \gamma_0$ and similarly $\mathbf{h}_0(x) = 0$ or $\mathbf{h}_0(x) = \beta_0$. Thus the following three cases are the only possible maps for $\mathbf{v}_0$ and $\mathbf{h}_0$ besides $\mathbf{v_0}(x) = 0 = \mathbf{h_0}(x)$. We will show that in each case we get a contradiction, and thus $\mathbf{v_0}(x) = 0 = \mathbf{h_0}(x)$ must hold.
    \paragraph*{Case 1:} Let $\mathbf{h}_0(x) = \beta_0$, $\mathbf{v}_0(x) = \gamma_0$. Here we have $$\mathbf{D}^+(x+a_1) = 2\beta_0 + 2 \gamma_0 = 0,$$ thus $x + a_1 \in A_0^{red}$. But $U\cdot (x + a_1) = a_0$, which contradicts Corollary \ref{genusthm}.

    \paragraph*{Case 2:} Let $\mathbf{h}_0(x) = \beta_0$, $\mathbf{v}_0(x) = 0$. Recall there exists an isomorphism $\phi \colon \mathbf{A}_0^+ \to \mathbf{A}_0^+$ that is $U$-equivariant, grading preserving, and satisfies $\mathbf{v}_0 \circ \phi = \mathbf{h}_0.$ Therefore $\mathbf{v}_0(\phi(x)) = \gamma_0$, and $\mathbf{h}_0(\phi(x)) = 0$. It follows that $\phi(x) \neq x$, and $U \cdot\phi(x) = 0$ by the $U$-equivariance of $\phi$. Thus $\phi(x)$ must not be in the tower. Then we have $$\mathbf{D}^+(a_1 + x + \phi(x)) = \beta_0 + \gamma_0 + \beta_0 + \gamma_0 = 0,$$ 
    thus $a_1 + x + \phi(x) \in A_0^{red},$ but $U \cdot (a_1 + x + \phi(x) \neq 0$, so this contradicts Corollary \ref{genusthm}.

    \paragraph*{Case 3:} Let $\mathbf{h}_0(x) = 0$, $\mathbf{v}_0 = \gamma_0$. This is argued similarly to case 2.

Thus we have shown that the only allowable maps for $\mathbf{v}_0$ and $\mathbf{h}_0$ must satisfy $\mathbf{v}_0(x) = 0 = \mathbf{h}_0(x)$, for any element $x$ in level 1 of $\mathbf{A}_0^{red}$. 
\end{proof}

\begin{prop}
Let $K$ be a genus 1 knot $K$ in $S^3$, and let $x$ be in level $i$ of $\mathbf{A}_0^{red}$. Then $\mathbf{v}_0(x_i) = 0 =\mathbf{h}_0(x_i)$ for $i \neq 0$.
\end{prop}

\begin{proof}
This follows directly from Lemmas \ref{lvls1} and \ref{lvls2}. 
\end{proof}

Given these restrictions, we are now ready to prove Theorem \ref{MCG1}. Figure \ref{1nsurg} depicts the truncated mapping cone for $\frac{1}{n}$ surgery of genus 1 knots for positive $n$. We consider the element $a_{(0,i)}^j$ to be in level $j$ of the $i$\textsuperscript{th} copy of $\mathbf{A}_0^+$, denoted $\mathbf{A}_{(0,i)}^{red}$. We have a similar notation for $b_{(0,k)}^l$ in level $l$ of  $\mathbf{B}_{0,k}$.

\begin{figure}[h]
\centering
\begin{tikzpicture}[thick,scale=0.75, every node/.style={scale=0.75}]
\node[hackennode] (a0) at (-5,4.4) {};
\node[hackennode] (a0) at (-5,4.6) {};
\node[hackennode] (a0) at (-5,4.8) {};
\node (label) at (-5,.25) {$\mathbf{A}_{(0,1)}^+$};
\filldraw[black] (-5,1) circle (2pt) node[anchor=west] {$a_{(0,1)}^0$};
\filldraw[black] (-5,2) circle (2pt) node[anchor=west] {$a_{(0,1)}^1$};
\filldraw[black] (-5,3) circle (2pt) node[anchor=west] {$a_{(0,1)}^2$};
\filldraw[black] (-5,4) circle (2pt) node[anchor=west] {$a_{(0,1)}^3$};
\draw (-5.5,1) rectangle (-6.5,4);
\node (label) at (-6,2.5) {$\mathbf{A}_{(0,1)}^{red}$};

\draw [blue, ->] (-5,0) -- (-4,-1); 
\node (label) at (-4.75,-.5) {$\mathbf{h}_0$};
\draw [blue, ->] (-2,0) -- (-1,-1); 
\node (label) at (-1.75,-.5) {$\mathbf{h}_0$};
\draw [blue, ->] (1,0) -- (2,-1); 
\node (label) at (1.25,-.5) {$\mathbf{h}_0$};
\draw [blue, ->] (6,0) -- (7,-1); 
\node (label) at (6.25,-.5) {$\mathbf{h}_0$};

\draw [red, ->] (-2,0) -- (-3,-1); 
\node (label) at (-2.8,-.5) {$\mathbf{v}_0$};
\draw [red, ->] (1,0) -- (0,-1); 
\node (label) at (.2,-.5) {$\mathbf{v}_0$};
\draw [red, ->] (9,0) -- (8,-1); 
\node (label) at (8.2,-.5) {$\mathbf{v}_0$};

\node[hackennode] (a0) at (-2,4.4) {};
\node[hackennode] (a0) at (-2,4.6) {};
\node[hackennode] (a0) at (-2,4.8) {};
\node (label) at (-2,.25) {$\mathbf{A}_{(0,2)}^+$};
\filldraw[black] (-2,1) circle (2pt) node[anchor=west] {$a_{(0,2)}^0$};
\filldraw[black] (-2,2) circle (2pt) node[anchor=west] {$a_{(0,2)}^1$};
\filldraw[black] (-2,3) circle (2pt) node[anchor=west] {$a_{(0,2)}^2$};
\filldraw[black] (-2,4) circle (2pt) node[anchor=west] {$a_{(0,2)}^3$};
\draw (-2.5,1) rectangle (-3.5,4);
\node (label) at (-3,2.5) {$\mathbf{A}_{(0,2)}^{red}$};

\node[hackennode] (a0) at (1,4.4) {};
\node[hackennode] (a0) at (1,4.6) {};
\node[hackennode] (a0) at (1,4.8) {};
\node (label) at (1,.25) {$\mathbf{A}_{(0,3)}^+$};
\filldraw[black] (1,1) circle (2pt) node[anchor=west] {$a_{(0,3)}^0$};
\filldraw[black] (1,2) circle (2pt) node[anchor=west] {$a_{(0,3)}^1$};
\filldraw[black] (1,3) circle (2pt) node[anchor=west] {$a_{(0,3)}^2$};
\filldraw[black] (1,4) circle (2pt) node[anchor=west] {$a_{(0,3)}^3$};
\draw (.5,1) rectangle (-.5,4);
\node (label) at (0,2.5) {$\mathbf{A}_{(0,3)}^{red}$};

\node[hackennode] (dot) at (3.6,2.5) {};
\node[hackennode] (dot) at (3.8,2.5) {};
\node[hackennode] (dot) at (4,2.5) {};

\node[hackennode] (a0) at (6,4.4) {};
\node[hackennode] (a0) at (6,4.6) {};
\node[hackennode] (a0) at (6,4.8) {};
\node (label) at (6,.25) {$\mathbf{A}_{(0,n-1)}^+$};
\filldraw[black] (6,1) circle (2pt) node[anchor=west] {$a_{(0,n-1)}^0$};
\filldraw[black] (6,2) circle (2pt) node[anchor=west] {$a_{(0,n-1)}^1$};
\filldraw[black] (6,3) circle (2pt) node[anchor=west] {$a_{(0,n-1)}^2$};
\filldraw[black] (6,4) circle (2pt) node[anchor=west] {$a_{(0,n-1)}^3$};
\draw (5.7,1) rectangle (4.3,4);
\node (label) at (5,2.5) {$\mathbf{A}_{(0,n-1)}^{red}$};

\node[hackennode] (a0) at (9,4.4) {};
\node[hackennode] (a0) at (9,4.6) {};
\node[hackennode] (a0) at (9,4.8) {};
\node (label) at (9,.25) {$\mathbf{A}_{(0,n)}^+$};
\filldraw[black] (9,1) circle (2pt) node[anchor=west] {$a_{(0,n)}^0$};
\filldraw[black] (9,2) circle (2pt) node[anchor=west] {$a_{(0,n)}^1$};
\filldraw[black] (9,3) circle (2pt) node[anchor=west] {$a_{(0,n)}^2$};
\filldraw[black] (9,4) circle (2pt) node[anchor=west] {$a_{(0,n)}^3$};
\draw (8.5,1) rectangle (7.5,4);
\node (label) at (8,2.5) {$\mathbf{A}_{(0,n)}^{red}$};

\node[hackennode] (a0) at (-3.5,-2.2) {};
\node[hackennode] (a0) at (-3.5,-2.4) {};
\node[hackennode] (a0) at (-3.5,.-2.6) {};
\node (label) at (-3.5,-1.5) {$\mathbf{B}_{(0,1)}^+$};
\filldraw[black] (-3.5,-6) circle (2pt) node[anchor=west] {$b_{(0,1)}^0$};
\filldraw[black] (-3.5,-5) circle (2pt) node[anchor=west] {$b_{(0,1)}^1$};
\filldraw[black] (-3.5,-4) circle (2pt) node[anchor=west] {$b_{(0,1)}^2$};
\filldraw[black] (-3.5,-3) circle (2pt) node[anchor=west] {$b_{(0,1)}^3$};
\draw (-4,-6) rectangle (-5,-3);
\node (label) at (-4.5,-4.5) {$\mathbf{B}_{(0,1)}^{red}$};

\node[hackennode] (a0) at (-.5,-2.2) {};
\node[hackennode] (a0) at (-.5,-2.4) {};
\node[hackennode] (a0) at (-.5,.-2.6) {};
\node (label) at (-.5,-1.5) {$\mathbf{B}_{(0,2)}^+$};
\filldraw[black] (-.5,-6) circle (2pt) node[anchor=west] {$b_{(0,2)}^0$};
\filldraw[black] (-.5,-5) circle (2pt) node[anchor=west] {$b_{(0,2)}^1$};
\filldraw[black] (-.5,-4) circle (2pt) node[anchor=west] {$b_{(0,2)}^2$};
\filldraw[black] (-.5,-3) circle (2pt) node[anchor=west] {$b_{(0,2)}^3$};
\draw (-1,-6) rectangle (-2,-3);
\node (label) at (-1.5,-4.5) {$\mathbf{B}_{(0,2)}^{red}$};

\node[hackennode] (a0) at (2.5,-2.2) {};
\node[hackennode] (a0) at (2.5,-2.4) {};
\node[hackennode] (a0) at (2.5,.-2.6) {};
\node (label) at (2.5,-1.5) {$\mathbf{B}_{(0,3)}^+$};
\filldraw[black] (2.5,-6) circle (2pt) node[anchor=west] {$b_{(0,3)}^0$};
\filldraw[black] (2.5,-5) circle (2pt) node[anchor=west] {$b_{(0,3)}^1$};
\filldraw[black] (2.5,-4) circle (2pt) node[anchor=west] {$b_{(0,3)}^2$};
\filldraw[black] (2.5,-3) circle (2pt) node[anchor=west] {$b_{(0,3)}^3$};
\draw (2,-6) rectangle (1,-3);
\node (label) at (1.5,-4.5) {$\mathbf{B}_{(0,3)}^{red}$};

\node[hackennode] (dot) at (5.6,-4.5) {};
\node[hackennode] (dot) at (5.4,-4.5) {};
\node[hackennode] (dot) at (5.2,-4.5) {};

\node[hackennode] (a0) at (7.5,-2.2) {};
\node[hackennode] (a0) at (7.5,-2.4) {};
\node[hackennode] (a0) at (7.5,.-2.6) {};
\node (label) at (7.5,-1.5) {$\mathbf{B}_{(0,n-1)}^+$};
\filldraw[black] (7.5,-6) circle (2pt) node[anchor=west] {$b_{(0,n-1)}^0$};
\filldraw[black] (7.5,-5) circle (2pt) node[anchor=west] {$b_{(0,n-1)}^1$};
\filldraw[black] (7.5,-4) circle (2pt) node[anchor=west] {$b_{(0,n-1)}^2$};
\filldraw[black] (7.5,-3) circle (2pt) node[anchor=west] {$b_{(0,n-1)}^3$};
\draw (7.2,-6) rectangle (5.8,-3);
\node (label) at (6.5,-4.5) {$\mathbf{B}_{(0,n-1)}^{red}$};
\end{tikzpicture}
\caption{$\frac{1}{n}$ surgery mapping cone for a genus 1 knot}
\label{1nsurg}
\end{figure}
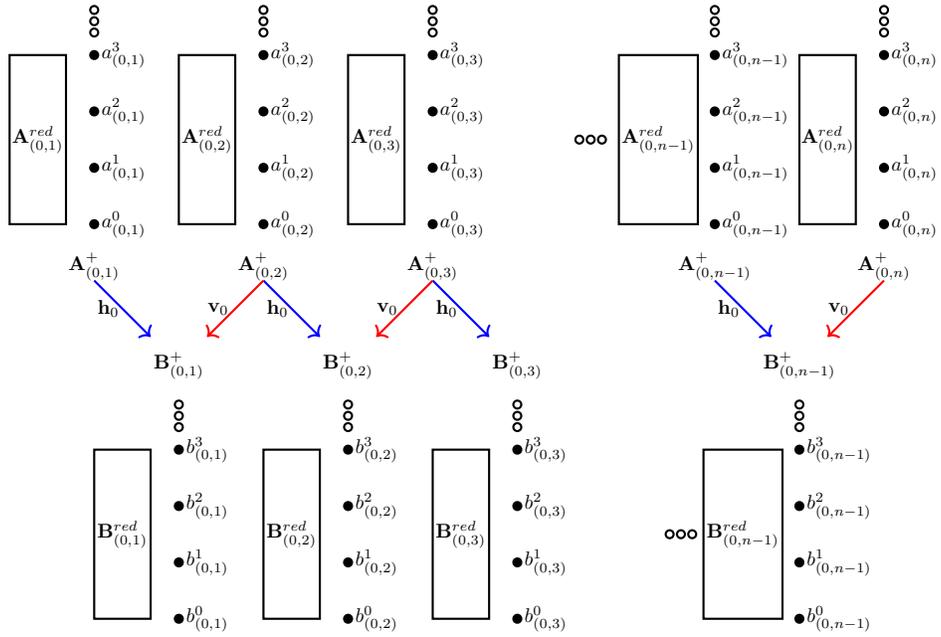

Now we are ready to prove our main result for genus 1 knots. 

\begin{proof}[Proof of Theorem \ref{MCG1}] By Proposition \ref{maprest}, the only elements in $\ker(\mathbf{D}^+_{1/n})$  consists of all elements of $\mathbf{A}_{(0,i)}^{red}$, elements of the form $a_{(0,1)}^i + a_{(0,2)}^i + \cdots + a_{(0,n)}^i$ and sums of these elements. We know that $a_{(0,1)}^i + a_{(0,2)}^i + \cdots + a_{(0,n)}^i$ is in the tower as $$U^k \left(a_{(0,1)}^{i+k} + a_{(0,2)}^{i+k} + \cdots + a_{(0,n)}^{i+k}\right) = a_{(0,1)}^i + a_{(0,2)}^i + \cdots + a_{(0,n)}^i.$$ Thus $HF_{red}(S^3_{1/n}(K)) = \mathbf{A}_0^{red}$ and so $U\cdot HF_{red}(S^3_{1/n}(K)) = 0$ by Corollary \ref{GainMC}, for positive $n$. \par
Now consider negative surgery. We know that $S^3_{1/n}(K) = -S^3_{-1/n}(\overline{K})$ where $\overline{K}$ is the mirror of $K$. Since $K$ and $\overline{K}$ are both genus 1 knots it follows that $$U \cdot HF_{red}(S^3_{-1/n}(K)) = 0$$ by the arguments for positive $\frac{1}{n}$ surgery above.
\end{proof}

\clearpage

\bibliography{References}

\end{document}